\documentclass[12pt]{amsart}

\usepackage{amssymb, color}
\usepackage[dvipsnames]{xcolor}
\usepackage{mathrsfs}
\usepackage{amsmath} 
\usepackage{amsthm}
\usepackage{url} 
\usepackage{stackrel} 
\usepackage{courier}
\usepackage[all,cmtip]{xy} 
\usepackage{multicol}
\usepackage{mathtools}
\usepackage{verbatim} % \begin{comment} \end{comment}
\usepackage{lipsum}
\usepackage{hyperref}
\usepackage[normalem]{ulem}
\usepackage[makeroom]{cancel}  
\usepackage{enumitem}
\catcode`~=11 % make LaTeX treat tilde (~) like a normal character
\newcommand{\urltilde}{\kern -.15em\lower .7ex\hbox{~}\kern .04em}  
\catcode`~=13 % revert back to treating tilde (~) as an active character

\newtheorem{thm}{Theorem}[subsection]

\newtheorem{prop}[thm]{Proposition}

\newcommand{\thmref}[1]{Theorem~\ref{#1}}
\newcommand{\propref}[1]{Proposition~\ref{#1}}

\newcommand{\eqnref}[1]{~(\ref{#1})}

\newcommand{\Inoindex}{\mathop{\mathrm{I}}\nolimits}

\newcommand{\tr}{\mathop{\mathrm{tr}}\nolimits}

\newcommand{\Aut}{\mathop{\mathrm{Aut}}\nolimits}

\newcommand{\g}{\mathop{\mathrm{g}}\nolimits}

\makeatletter
\newif\if@borderstar
   \def\bordermatrix{\@ifnextchar*{%
       \@borderstartrue\@bordermatrix@i}{\@borderstarfalse\@bordermatrix@i*}%
   }
   \def\@bordermatrix@i*{\@ifnextchar[{\@bordermatrix@ii}{\@bordermatrix@ii[()]}}
   \def\@bordermatrix@ii[#1]#2{%
   \begingroup
     \m@th\@tempdima8.75\p@\setbox\z@\vbox{%
       \def\cr{\crcr\noalign{\kern 2\p@\global\let\cr\endline }}%
       \ialign {$##$\hfil\kern 2\p@\kern\@tempdima & \thinspace %
       \hfil $##$\hfil && \quad\hfil $##$\hfil\crcr\omit\strut %
       \hfil\crcr\noalign{\kern -\baselineskip}#2\crcr\omit %
       \strut\cr}}%
     \setbox\tw@\vbox{\unvcopy\z@\global\setbox\@ne\lastbox}%
     \setbox\tw@\hbox{\unhbox\@ne\unskip\global\setbox\@ne\lastbox}%
     \setbox\tw@\hbox{%
       $\kern\wd\@ne\kern -\@tempdima\left\@firstoftwo#1%
         \if@borderstar\kern2pt\else\kern -\wd\@ne\fi%
       \global\setbox\@ne\vbox{\box\@ne\if@borderstar\else\kern 2\p@\fi}%
       \vcenter{\if@borderstar\else\kern -\ht\@ne\fi%
         \unvbox\z@\kern-\if@borderstar2\fi\baselineskip}%
         \if@borderstar\kern-2\@tempdima\kern2\p@\else\,\fi\right\@secondoftwo#1 $%
     }\null \;\vbox{\kern\ht\@ne\box\tw@}%
   \endgroup
   }
\makeatother

\newtheorem{theorem}{Theorem}[section] 

\newtheorem{lemma}[theorem]{Lemma}

\newtheorem{example}[theorem]{Example}

\newtheorem{remark}[theorem]{Remark}

\newtheorem{corollary}[theorem]{Corollary}

\renewcommand{\a}{\alpha}

\def\g{\gamma}

\def\e{\epsilon}

\def\C{\mathbb{C}}
\def\Z{\mathbb{Z}}
\def\N{\mathbb{N}}

\def\G{\mathbb{G}}
\def\U{\mathbb{U}}
\def\V{\mathbb{V}}

\newcommand\Dic{\mathrm{Dic}}

\allowdisplaybreaks 

\begin{document}

\title[The center of hyperelliptic Krichever-Novikov algebras]{On the module structure of the center of hyperelliptic Krichever-Novikov algebras II} 
\author{Ben Cox}

\address{Department of Mathematics,  College of
Charleston  \\
Charleston, SC 29424 USA
}
\email{coxbl@cofc.edu}

\author{Xiangqian Guo}
\address{Department of Mathematics,  Zhengzhou
University \\
 Zhengzhou 450001, Henan, P. R. China
 }
\email{guoxq@zzu.edu.cn}

\author{Mee Seong Im}
\address{Department of Mathematical Sciences,  United States Military Academy \\
 West Point, NY 10996 USA 
}
\email{meeseongim@gmail.com}

\author{Kaiming Zhao}
\address{Department of Mathematics, Wilfrid
Laurier University \\  Waterloo, ON, Canada N2L 3C5 \\ 
and College of
Mathematics and Information Science, Hebei Normal (Teachers)  
University, Shijiazhuang, Hebei, 050016 P. R. China 
}
\email{kzhao@wlu.ca}

\maketitle

\begin{abstract} 
Let $R := R_{2}(p)=\mathbb{C}[t^{\pm 1}, u :  u^2 = t(t-\alpha_1)\cdots (t-\alpha_{2n})] $ be the coordinate ring of a nonsingular hyperelliptic curve and let $\mathfrak{g}\otimes R$ be the corresponding current Lie algebra. \color{black} Here $\mathfrak g$ is a finite dimensional simple Lie algebra defined over $\mathbb C$ and  
\begin{equation*}
p(t)= t(t-\alpha_1)\cdots (t-\alpha_{2n})=\sum_{k=1}^{2n+1}a_kt^k.
\end{equation*}
In earlier work, Cox and Im gave a generator and relations description of the universal central extension of $\mathfrak{g}\otimes R$ in terms of certain families of polynomials $P_{k,i}$ and $Q_{k,i}$ and they described how the center $\Omega_R/dR$ of this universal central extension decomposes into a direct sum
of irreducible representations when the automorphism group was the cyclic group $C_{2k}$ or the dihedral group $D_{2k}$. 
 We give examples of $2n$-tuples $(\alpha_1,\dots,\alpha_{2n})$, which are the  automorphism groups  %  $\mathbb V_{2n}$, 
 $\mathbb G_n=\text{Dic}_{n}$,  $\mathbb U_n\cong D_n$ ($n$ odd), or $\mathbb U_n$ ($n$ even)  of the hyperelliptic curves 
 \begin{equation}
S=\mathbb{C}[t, u: u^2 = t(t-\alpha_1)\cdots (t-\alpha_{2n})] 
 \end{equation}
 given in \cite{MR3631928}. 
  In the work below, we describe this decomposition when the automorphism group is 
 % $\mathbb V_{2n}$, 
 % $\mathbb G_n=\text{Dic}_{n}$, and 
 $\mathbb U_n=D_n$, where $n$ is odd. 
 % or $\mathbb U_n$ ($n$ even). 
\end{abstract}  
%\thanks{The first author  }
%\setlength{\smallskipamount}{6pt} 
%\setlength{\medskipamount}{10pt} 
%\setlength{\bigskipamount}{16pt}

\maketitle 

\bibliographystyle{amsalpha} 

\setcounter{tocdepth}{0} 

%\tableofcontents 

% Check the coefficients: the roots \alpha_i  versus the poly coeff a_i's.  
 
\section{Introduction}\label{section:intro} 

% {\color{red} \textbf{Kaiming}, please read the introduction and are you satisfied with the way that it is currently written? Are there any information that has been omitted? Missing references?} 

In  \cite{MR3211093} amongst other results, the automorphism group of derivation ring $\text{Der}(R)$ for $R=\mathbb{C}[t,(t-a_1)^{-1},\ldots, (t-a_{n})^{-1}]$ was described and interestingly enough the five Kleinian groups $\Z_n$, $D_n$, $A_4$, $S_4$ and $A_5$  appear as automorphism groups of $R$ for particular choices of $a_1,\dots, a_n$. These five groups also naturally appear in the McKay correspondence, tying together the representation theory of finite subgroups $G$ of $SL_n(\mathbb{C})$ to the resolution of singularities of quotient orbifolds $\mathbb{C}^n/G$.  
  In this same paper, \cite{MR3211093}, the authors also described the universal central extension of this derivation Krichever-Novikov algebra.  One can in a straightforward manner show that the automorphism group acts on the center $\Omega_R/dR$ and one can naturally ask how the center decomposes into a direct sum of irreducible representations.   This is exactly what described in \cite{cox-center-genus-zero-KN-algebras}.

In previous work (see \cite{MR3845909})  of the first and third authors we described how the center $\Omega_R/dR$ of the universal central extension of the hyperelliptic Lie algebra $\mathfrak{g}\otimes R$ with $R=\mathbb C[t,t^{-1},u:  u^2=p(t)]$ decomposes as a direct sum of irreducible modules for the corresponding automorphism groups $C_n$ and $D_n$, ($n$ even).  Here $p(t)=t(t-\alpha_1)\cdots (t-\alpha_{2n})=\sum_{i=1}^{2n+1}a_it^i$, with $\alpha_i$ pairwise distinct complex numbers\color{black}.   From the work of Grothendieck on the relationship between a Riemann surface and its de Rham cohomology groups one has
\begin{equation}\label{omegaRmoddR}
\Omega_R/dR=\mathbb C \omega_0\oplus \ldots \oplus \mathbb C\omega_{2n}
\end{equation}
where $\omega_0=\overline{t^{-1}\,dt}$ and $\omega_i=\overline{t^{-j}u\,dt}$ for $i=1,\dots, 2n$.
 
 In \cite{MR3845909} we described how $\Omega_R/dR$ decomposes as a sum of irreducible representations for the groups $C_{2k}$ and $D_{2k}$ when $k|n$.

%A description of how the center of the universal central extension of the genus zero Krichever-Novikov algebra decomposes under the action of the above automorphism groups into a direct sum of irreducible representations was studied  in \cite{cox-center-genus-zero-KN-algebras}. 
%

Using this theorem one can give a description of the bracket of two basis elements in the universal central extension of $\mathfrak g\otimes R$ in terms of polynomials $P_{k,i}$ and $Q_{k,i}$ defined recursively
\begin{equation} \label{}
(2k+r+3)P_{k,i}=-\sum_{j= 1 }^{r} (3j+2k-2r)a_jP_{k-r+j-1,i} \quad \mbox{ for all } k\geq 0 
\end{equation}
 with the initial conditions
$P_{l,i}=\delta_{l,i}$, $-r \leq i,l\leq -1$ 
and 
\begin{equation} 
 ( 2m-3 )a_1Q_{m,i}  = \sum_{j=2}^{r+1 } ( 3j -2m)a_j Q_{ m - j+ 1 ,i}  \quad \mbox{ for all }  m>0    
\end{equation}
with the initial conditions $Q_{l,i}=\delta_{l,-i}$ for $1\leq l\leq r$ and $-r\leq i\leq -1$.  In this paper $\mathfrak g$ is assumed throughout to be a finite dimensional simple Lie algebra defined over the complex numbers. 
The generating series for these polynomials can be written in terms of hyperelliptic integrals using Bell polynomials and Fa\'a de Bruno's formula (see \cite{Arbogast}, \cite{MR1502817}, \cite{dB1} and \cite{dB2}).   One can compare this result to that given in \cite{MR3487217} and also in \cite{coxzhao-2016}.

We also describe in this paper (see \thmref{theorem-result-001}) how K\"ahler differentials modulo exact forms $\Omega_R/dR$ decompose under the action of the automorphism group of the coordinate ring 
% The general setting as R = R_m(p) = \mathbb{C}[t^{\pm 1}, u : u^m = p(t)] 
$R := R_{2}(p)=\mathbb{C}[t^{\pm 1}, u : u^2 = p(t)]$ , where $p(t)=t(t-\alpha_1)\cdots (t-\alpha_{2n})=\sum_{i=1}^{2n+1}a_it^i$, with the $\alpha_i$ being pairwise distinct roots.
We recall Theorem~\ref{cor:group-actions}, where we describe the automorphism group of the universal central extension of the hyperelliptic Lie algebra as a $\mathbb{Z}_2$-graded Lie algebra.  This corrects a small error in our previous work \cite{MR3631928}. 
 
 The possible automorphism groups for the hyperelliptic curve 
$$
R =\mathbb{C}[t^{\pm 1}, u : u^2 = t(t-\alpha_1)\cdots (t-\alpha_{2n})] 
$$
are the groups  $C_{2k}=\Z_{2k}$ or one of the groups
%\begin{align*}
%\mathbb V_{2k}&:=\langle x,y\,|\,x^4,y^{2k},(xy)^2,(x^{-1}y)^2\rangle, \\
%\mathbb U_k&:=\langle x,y\,|\,x^2,y^{2n},xyxy^{k+1} \rangle\\
%Dic_k&:=\langle a,x\,|\, a^{2k}=1,\enspace x^2=a^k,\enspace x^{-1}ax=a^{-1}\rangle
%\end{align*} 
\begin{align*} 
D_{2k}&:=\langle x,y\, |\, x^{2k}=1, y^{2}=1, xyx=y\rangle,\\ 
Dic_k&:=\langle x,y\,|\, x^{2k}=1, y^2=x^k, xyx=y\rangle,\\ 
\mathbb U_k&:=\langle x,y\,|\,x^{2k}=1,y^{2}=1,xyx=yx^k \rangle.  \\ 
\end{align*}
See \thmref{cor:group-actions} below, \cite[Corollary 15]{MR3631928}, \cite{MR1223022} and \cite{MR2035219} for more detail.

The polynomials $P_{k,i}$ and $Q_{l,j}$, which are defined in Proposition $3.3$ in \cite{MR3845909}, give us a description of how the center decomposes under the group of automorphisms of $R$.   The automorphism group of $R$ has a  natural action on $\Omega_R/dR$ and thus it is natural to wonder how this representation decomposes into a direct sum of irreducible representations.    
When the automorphism group is $C_{2k}$, we can rewrite \eqnref{omegaRmoddR} as a direct sum of $1$-dimensional irreducible $C_{2k}$-representations.
 More precisely the center decomposes as: 
\begin{equation}
\Omega_R/dR\cong U_0\oplus \ldots \oplus U_{k-1}, 
 \end{equation}  
 where $\displaystyle{U_r=\bigoplus_{i\equiv r\mod k,1\leq i\leq 2n}\mathbb C\omega_i}$ for $r=1,\ldots, k-1$, a sum of one-dimensional irreducible representation of $C_{2k}$.  
The $U_r$ are irreducible representations with character $\chi_r(s)=\exp(2\pi\imath rs/2k)$ each occurring with multiplicity $l$ and  
 \[ 
 U_0=\mathbb C\omega_0\oplus \bigoplus_{i=1}^l \mathbb C\omega_{ki}. 
 \]

If the automorphism group is $D_{2k}$, with $c^{2n}=a_1$, $k|2n$, and $k$ is even, 
% MSI:  I added the condition   $k|n$  to $n$  even.  
the center decomposes under the action of $D_{2k}$ as
\begin{equation} 
\Omega_R/dR \cong \mathbb{C} \omega_0 \oplus  \bigoplus_{i=3}^{4} U_i^{\frac{(1-(-1)^k)n}{2k} } \oplus \bigoplus_{h=1}^{k-1} V_{h}^{\oplus \frac{(1-(-1)^h)n}{k}}. 
\end{equation} 
where $U_i$, $i=1,2,3,4$, are the irreducible one dimensional representations for $D_{2k}$ with character $\rho_i$ and $V_h$ are the irreducible 2-dimensional representations for $D_{2k}$ with character $\chi_h$, $1\leq h\leq k-1$.   
Here $\mathbb C\omega_0$ and $U_1$ are the trivial representation.

On the other hand if $k$ is odd, then the center decomposes as 
$$
\Omega_R/dR \cong \mathbb{C}\omega_0 
\oplus \bigoplus_{i=3}^{4} U_i^{\oplus \Upsilon_i(\epsilon_i, \nu_i)} 
\oplus 
\bigoplus_{h=1}^{k-1} V_h^{\oplus \frac{(1-(-1)^h)n}{k}}, 
$$ 
where 
\begin{align*} 
\Upsilon_i(\epsilon_i,\nu_i) 
%&= \frac{(1-(-1)^k)n}{2k}(\delta_{i,3}+\delta_{i,4})+\frac{\epsilon_i }{4}   \displaystyle{\sum_{i=n+3}^{2n}}
%c^{n+3-2i}
%  P_{i-n-3,-i} + \frac{\nu_i}{4}\displaystyle{\sum_{i=n+3}^{2n}}
%\zeta^{\frac{2n+3-2i}{2}}
%c^{n+3-2i}
%  P_{i-n-3,-i} \\ 
  &= \frac{(1-(-1)^k)n}{2k}(\delta_{i,3}+\delta_{i,4})+(-1)^i\frac{1-(-1)^n}{4} +\frac{1 }{2} (-1)^i  \displaystyle{\sum_{i=n+3}^{2n}}
c^{n+3-2i}
  P_{i-n-3,-i}.
\end{align*} 
This is all derived in \cite[Theorem 7.2]{MR3845909}.

We now turn to the case when the automorphism group is $D_{n}$, where $n$ is odd. 
For $n$ odd, the center decomposes as  
\begin{equation} 
\Omega_R/dR \cong \mathbb{C} \omega_0 \oplus U_1^{\Xi_1} 
\oplus U_2^{\Xi_2} \oplus \bigoplus_{j=1}^{\frac{n-1}{2}} V_{j},  
\end{equation}
where   
\begin{align*}
\Xi_1 = \frac{1}{2}- \frac{1}{2} \sum_{i=n+3}^{2n} c^{\frac{n+3-2i}{2}}P_{i-n-3,-i} 
\qquad 
\mbox{ and }  
\qquad 
\Xi_2 = \frac{3}{2}+ \frac{1}{2} \sum_{i=n+3}^{2n} c^{\frac{n+3-2i}{2}}P_{i-n-3,-i}, 
\end{align*}
$\mathbb{C} \omega_0$ 
is a $1$-dimensional irreducible representation, $U_i$ are pairwise distinct $1$-dimensional irreducible representations, 
and 
$V_j$ are pairwise distinct $2$-dimensional irreducible representations (see \thmref{theorem-result-001} below).  
To prove this result below, we use the representation theory techniques of Frobenius and Schur as explained by Serre in \cite{MR0450380} and by Fulton and Harris in \cite{MR1153249} to prove our results.

Since $\text{Dic}_k$ for $k$ gives the same character table as that for $D_{2k}$ in our paper the decomposition theorem should nearly look the same with the exception of that the irreducible representations that appear are different and are for different groups. In other words, the multiplicities of the irreducibles are the same but the representations are different.  We will explore this in our future work. 
 
 \color{black}

\section{Automorphism group for $R=\mathbb C[t^{\pm 1},u : u^2=p(t)=t(t-\alpha_1)\cdots (t-\alpha_{2n})]$.}\label{section:automorphism}
 
In this section, we restrict to the case of $r=2n$ which allows us to use the results in  \cite{MR3631928}, \cite{MR1223022} and \cite{MR2035219} on automorphism groups of such algebras. 
 \subsection{Automorphisms of the algebra $R_2(p)$} 
Let $S_{2n}$ be the symmetry group on the finite set $\{1,2,\ldots, 2n \}$. 

First we recall some background material.   

\begin{theorem}[\cite{MR1223022} and \cite{MR2035219}] \label{BGG}
The automorphism group of a hyperelliptic curve $A=\mathbb C[X,Y|Y^2=P(X)]$ is isomorphic to one of the following groups:
$$
D_n, \quad 
\mathbb Z_n, \quad 
\mathbb V_n,\quad  
\mathbb H_n,\quad  
\mathbb G_n, \quad 
\mathbb U_n, \quad 
GL_2(3),\quad 
W_2, \quad  W_3
$$
where
\begin{align*}
\mathbb V_n&:=\langle x,y\,|\,x^4,y^n,(xy)^2,(x^{-1}y)^2\rangle, \\
\mathbb H_n&:=\langle x,y\,|\,x^4,y^2x^2,(xy)^n\rangle ,\\
\mathbb G_n&:=\langle x,y\,|\,x^2y^n,y^{2n},x^{-1}yxy\rangle, \\
\mathbb U_n&:=\langle x,y\,|\,x^2,y^{2n},xyxy^{n+1} \rangle, \\
W_2&:=\langle x,y|x^4,y^3,yx^2y^{-1}x^2,(xy)^4\rangle, \\
W_3&:=\langle x,y|x^4,y^3,x^2(xy)^4,(xy)^8\rangle.
\end{align*}
\end{theorem}
In \cite{MR2035219} a description of the reduced automorphism group is described for a given polynomial $P(X)$.  In our paper  we do not work with the reduced automorphism group and our coordinate ring is the localization $\mathbb C[t,t^{-1},u : u^2=p(t)=t(t-\alpha_1)\cdots (t-\alpha_{2n})]$ of $A$.

For convenience, we need the following alternative description of the group $\U_n$:
\begin{equation}\label{U'}
\mathbb U_n=\langle x_1,y_1 : y_1^{2n}=1,x_1^{2}=y_1^n, y_1x_1y_1=x_1^{-1}\rangle.
\end{equation}
Indeed, setting $y=y_1$ and $x=x_1y_1$, we can get $y^{2n}=1$, $x^2=x_1y_1x_1y_1=1$ and 
$xyxy^{n+1}=x_1y^2_1x_1y_1^{n+2}=y_1^{-1}(y_1x_1y_1)^2y_1^{n+1}=y_1^{-1}x_1^{-2}y_1^{n+1}=1$,
which is just the generating relations of $\U_{n}$ in Theorem \ref{BGG}. 
On the other hand, if $x, y$ are the generators of $\U_n$ as described in Theorem \ref{BGG},
we set $y_1=y$ and $x_1=xy$, which gives $y_1^{2n}=1$, $x_1^2=xyxy=y^n=y_1^n$ and
$y_1x_1y_1=yxyy=xy^{n+1}=x_1y_1^n$, the generating relations in \eqref{U'}.
So the group $\U_n$ defined by \eqref{U'} coincides with the group $\U_n$ in Theorem \ref{BGG}.

The result below describes the   automorphisms of the algebra of the superelliptic curve $u^2=p(t)$, and corrects some errors that occur in \cite{MR3631928}, Corollary 15.  

 \begin{theorem}[Corollary 15, \cite{MR3631928}]\label{cor:group-actions}  % group action 
Let $p(t)=t(t-\alpha_1)\cdots (t-\alpha_{2n})$, where $\alpha_i$ are pairwise distinct nonzero  roots. 
The only possible automorphisms $\phi\in \Aut(R_2(p))$ of the algebra $R_2(p)$ are of the following types: 
 \begin{enumerate}
\item\label{item:cor-cyclic-thm} 
There exist some $4n$-th root of unity $\xi$ and $\gamma\in S_{2n}$ with  $\alpha_{\gamma(i)}=\xi^2\alpha_i$ such that
 \begin{equation}\label{eq:cyclic-Z2-case}
 \phi(t)=\xi^2t,\quad \phi(u)=\xi u. 
 \end{equation}
We denote this $\phi$ by $\phi_{\xi}$ which satisfies $(\phi_{\xi})^{4n}=\text{id}$.
Note that  $\g$  is uniquely determined by $\xi^2$.
\item\label{cor:group-action-dihedral-maps-thm} There exists $c\in\mathbb C$ and $\gamma\in S_{2n}$ with 
$\alpha_i \alpha_{\gamma(i)}=c^2$ such that  
\begin{equation}\label{psia}
\phi(t)=c^2t^{-1}, \,\,\,  \phi(u)=\e c^{n+1}t^{-n-1}u,
\end{equation} 
where $\e=\pm1$ if $\prod_{i=1}^{2n}\alpha_i=c^{2n}$ and $\e=\pm\imath$ if $\prod_{i=1}^{2n}\alpha_i=-c^{2n}$.
We denote  these $\phi$ by $\psi^{\pm}_{c}$ respectively which satisfy $(\psi^{\pm}_{c})^2=\text{id}$ if $\e=\pm1$; $(\psi_{c}^\pm)^2=\phi_{-1}$ if $\e=\pm\imath$; $\psi^{-}_{c}=\psi^{+}_{-c}$ if $n$ is even; and 
$\psi^{\pm}_{-c}=\psi^{\pm}_{c}$ if $n$ is odd.
Note that  $\g$ is uniquely determined by $c^2$.  
\end{enumerate} 
% then $\psi^{-}_{c}=\psi^{+}_{-c}$ and $\psi^{-}_{-c}=\psi^{+}_{c}$ if $n$ is even, so we do not need consider $\psi^-_{c}$ but all possible of $c$ in this case; and $\psi^{+}_{-c}=\psi^{+}_{c}$ and $\psi^{-}_{-c}=\psi^{-}_{c}$ if $n$ is odd, so we need consider $\psi^-_{c}$ but half number of all possible $c$ in this case. 
\end{theorem}

For convenience we denote $\psi_{c}=\psi^{+}_{c}$ in Theorem \ref{cor:group-actions}(2), and let $\Aut_1(R_2(p))$ be the set of all automorphisms coming from \ref{cor:group-actions} (1),
which is a subgroup of $\Aut(R_2(p))$.
 
\begin{corollary}[\cite{MR3631928}, Corollary 16]\label{cor:automorphism-groups}
Let $p(t)=t(t-\alpha_1)\cdots (t-\alpha_{2n})$, with pairwise  distinct roots. 
 \begin{enumerate}
\item\label{item:MR3631928-1} There exists some $k\in\N$ with $k|2n$ such that $\Aut_1(R_2(p))$ is generated by an automorphism 
$\phi_{\xi}$ of order $2k$, where $\xi$ is any primitive root of unity of order $2k$.
\item\label{item:MR3631928-2} If $\psi_c$ does not exist in $\Aut(R_2(p))$ for any nonzero complex number $c$, then 
$$\Aut(R_2(p))=\Aut_1(R_2(p))=\langle \phi_{\xi}\rangle \simeq \Z_{2k},$$
where $k$ and $\xi$ are as in \eqref{item:MR3631928-1}.
\item\label{item:MR3631928-3}  If $\psi_c$ exists in $\Aut(R_2(p))$ for some complex number $c$ with $c^{4n}=\prod_{i=1}^{2n}\a_i$, then 
\begin{equation}\label{Aut}
\Aut(R_2(p))=\{\phi_\xi^i, \psi_c\phi_\xi^i :  i=0,1,\ldots,2k-1\},
\end{equation}
where $\phi_{\xi}\in\Aut_1(R_2(p))$ is an automorphism of order $2k$ as in \eqref{item:MR3631928-1}.
In particular, we have $|\Aut(R_2(p))|=4k$. Moreover we have:
\begin{equation}
\Aut(R_2(p)) =
\begin{cases}
D_k  & \mbox{ if } 2n/k\ \text{is\ even\ and}\ \prod_{i=1}^{2n}\a_i=c^{2n}, \\ 
\G_k &\mbox{ if } 2n/k\ \text{is\ even\ and}\ \prod_{i=1}^{2n}\a_i=-c^{2n}, \\ 
\mathbb U_k\mbox{ with }k \mbox{ even }  &\mbox{ if }   2n/k\ \text{is\ odd}. \\ 
\end{cases}  
\end{equation}
\end{enumerate}
\end{corollary}

\begin{proof} 
For any $\phi_{\xi}, \phi_{\xi'}\in \Aut(R_2(p)), \xi, \xi'\in\C$, we have 
$\phi_{\xi}\phi_{\xi'}=\phi_{\xi\xi'}$ and hence $\Aut_1(R_2(p))$ is a finite abelian group. 
Let $\phi_\xi$ be an automorphism of the greatest order $2k$ with $k|2n$. 
Then any other automorphism $\phi_{\xi'}$ must have order $k'$ with $k'|k$. 
Hence $\xi'=\xi^{r}$ for some $r\in\N$ and $\phi_{\xi'}=\phi_\xi^r$. 
Now $\Aut_1(R_2(p))$ is generated by $\phi_\xi$. Hence \eqref{item:MR3631928-1} and \eqref{item:MR3631928-2} follow.

To prove \eqref{item:MR3631928-3}, we first note that $\psi_c\psi_{c'}\in\Aut_1(R_2(p))$ for any $\psi_c, \psi_{c'}\in\Aut(R_2(p))$, where $c,c'\in\C$. 
Let $\xi, \phi_\xi$ and $k$ be as in \eqref{item:MR3631928-1}, %$\phi_\xi$ be a generator of $\Aut_1(R_2(p))$ that has oder $2k$ with $k|2n$. 
then $\xi$ is a primitive root of unity of order $2k$ and hence $\xi^k=-1$. Set $l=2n/k$.
Take any $\psi_c\in\Aut(R_2(p))$. Now we can deduce \eqref{Aut} easily.
Noticing that elements listed in  \eqref{Aut} are pairwise distinct, we see $|\Aut(R_2(p))|=4k$.

With  straightforward computations we deduce that 
$$(\phi_\xi\psi_c\phi_\xi)(t)=\psi_c(t)\ \text{and}\ (\phi_\xi\psi_c\phi_\xi)(u)=\xi^{-2n}\psi_c(u)=(-1)^{l}\psi_c(u).$$ %=\e \xi^{-2n}c^{n+1}t^{-n-1}u$$

In case $l$ is even, we have $\phi_\xi\psi_c\phi_\xi=\psi_c$.
If $\prod_{i=1}^{2n}\a_i=c^{2n}$, then $\psi_c^2=\mathrm{id}$ and hence we have an epimorphism from 
$D_{2k}$ to $\text{Aut}(R_2(p))$. Since $|D_{2k}|=4k$, we have $\text{Aut}(R_2(p))\cong D_{2k}$.
Similarly, if $\prod_{i=1}^{2n}\a_i=-c^{2n}$, then $\psi_c^2=\phi_{-1}=\phi_\xi^k$ and 
we have an epimorphism from $\G_{k}$ to $\text{Aut}(R_2(p))$. 
Moreover,  from the fact $|\G_{k}|\leq 4k$ we obtain that  $\text{Aut}(R_2(p))\cong \mathbb G_k=\text{Dic}_k$.

In case $l$ is odd, we have $\phi_\xi\psi_c\phi_\xi=\psi_{c}\phi_{-1}=\psi_{c}\phi_{\xi}^k$.
If $\prod_{i=1}^{2n}\a_i=c^{2n}$, then $\psi_c^2=\mathrm{id}$ and $\psi_c, \phi_\xi$ satisfy the generating relations for $\U_k$ in Theorem \ref{BGG}; if $\prod_{i=1}^{2n}\a_i=-c^{2n}$, then $\psi_c^2=\phi_{-1}=\phi_\xi^k$ and $\psi_c, \phi_\xi$ satisfy the generating relations for $\U_k$ in \eqref{U'}. 
We get an epimorphism from $\U_{k}$ ($k$ even) to $\text{Aut}(R_2(p))$, which is an isomorphism by a similar argument as in the previous paragraph.
%and $\text{Aut}(R_2(p))\cong \mathbb U_k$ with $k$ being even; 
%if $\prod_{i=1}^{2n}\a_i=-c^{2n}$,
%then $\psi_c^2=\phi_{-1}=\phi_\xi^k$ and $\text{Aut}(R_2(p))\cong\mathbb V_{2k}$ again with $k$ being even.
%The arguments are similar as above.
\end{proof}

%Where 
%\begin{align*}
%D_{2k}&:=\langle x,y\, |\, y^{2k}=1, x^{2}=1, yxy=x\rangle=\{y^i,xy^i\ |\ i=0,1,\ldots,2k-1\},\\
%%Dic_k&:=\langle x,y\,|\, x^{2k}=1, y^2=x^k, xyx=y\rangle=\{x^i,x^iy\ |\ i=0,1,\cdots,2k-1\},\\
%Dic_k=\mathbb G_k&:=\langle x,y\,|\, y^{2k}=1, x^2=y^k, yxy=x\rangle \subseteq\{y^i, xy^i\ |\ i=0,1,\cdots,2k-1\},\\
%              %&=\langle x,y\,|\,y^{2k}=1,x^{2}y^k=1, x^{-1}yxy=1\rangle\\
%\mathbb U_k&:=\langle x,y\,|\,y^{2k}=1,x^{2}=1,yxy=xy^k \rangle \subseteq \{y^i,xy^i\ |\ i=0,1,\cdots,2k-1\},\\
%%\mathbb U_k&:=\langle x,y\,|\,x^{2k}=1,y^{2}=1,xyx=yx^k \rangle=\{x^i,yx^i\ |\ i=0,1,\cdots,2k-1\},\\
%\mathbb U_{k}&:=\langle x,y\,|\,y^{2k}=1,x^{2}=y^k,yxy=x^{-1}\rangle \subseteq\{y^i,xy^i\ |\ i=0,1,\cdots,2k-1\}\\
% \leftarrow V_{2k}& =\langle x,y\,|\,y^{2k}=1,x^{4}=1,xyxy=1, x^{-1}yx^{-1}y=1\rangle\\
%%\mathbb V_{2k}&:=\langle x,y\,|\,x^{2k}=1,y^{2}=x^k,xyx=y^{-1}\rangle=\{x^i,yx^i\ |\ i=0,1,\cdots,2k-1\}\\
% %             &?=\langle x,y\,|\,x^{2k}=1,y^{4}=1,yxyx=1, y^{-1}xy^{-1}x=1\rangle\\
%\end{align*}
%$$iso: \mathbb U_k\cong D_k,\ \text{if\ } k\ \text{is\ odd};$$
%$$iso: \mathbb V_{2k}\rightarrow \mathbb U_{k},\ \ y\mapsto y, x\mapsto xy,\ $$%\text{at\ least\ for\ }k=2.$$
%%
In the next example we will realize the groups $\Dic_n=\G_n$ and $ \U_n$.

\begin{example} 
 Let $p(t)=t\prod_{i=1}^l\prod_{j=1}^{k}(t-c_i\xi^{2j})$ with pairwise distinct roots, where $\xi$ is primitive root of unity of order $2k$ and $l=2n/k$. Set $\a_{(i-1)k+j}=c_i\xi^{2j}$. We have $\ \phi_{\xi}\in\Aut_1(R_2(p))$ which has order $2k$.
\begin{itemize}
\item[(1)] Suppose that $l=3$ and $|c_1|, |c_2|,|c_3|$ are pairwise distinct. We see that $\Aut_1(R_2(p))$ $=\langle \phi_{\xi}\rangle\cong \Z_{2k}$.
 But we can not find $c\in\C$ and $\gamma\in S_{2n}$
such that $\a_i\a_{\gamma(i)}=c^2$ (i.e., $|c|^2=|c_i|\cdot |c_{i'}|$ for all $i=1, 2,3$ where $i'\ne i$ is determined by $i$). So there does not exist automorphism of $R_2(p)$ coming from (2) of Lemma \ref{cor:group-actions}. Hence $\Aut(R_2(p))=\Aut_1(R_2(p))\cong \Z_{2k}$.
\item[(2)] Suppose that $l=2$ and $|c_2|\neq |c_1|$. Then $k=n$, $\xi^n=-1$ and $\prod_{i=1}^{2n}\a_i=c_1^{n}c_2^n$. Take any $c\in\C$ such that $c^2=c_1c_2$, we have $\prod_{i=1}^{2n}\a_i=c^{2n}$ and 
$$
c_1\xi^{j-i}c_2\xi^{j+i}=(\xi^j c)^2, \quad \prod_{i=1}^{2n}\a_i=(\xi^j c)^{2n}.
$$ 
We can define the automorphisms $\psi^{\pm}_{\xi^jc}$ for all $j=0,1,\ldots,2n-1$, that is,
$$\psi^{\pm}_{\xi^jc}:  t\mapsto \xi^{2j}c^2t^{-1}, 
\qquad
 \phi(u)=\pm\xi^{j(n+1)}c^{n+1}t^{-n-1}u.$$
Denote $\psi_c=\psi^+_{c}$. We can check that 
$\psi^{+}_{\xi^jc}=\psi_c\phi_\xi^{j}$ if $j$ is even and
$\psi^{-}_{\xi^jc}=\psi_c\phi_\xi^{j}$ if $j$ is odd.
Moreover,
$\psi^{-}_{\xi^jc}=\psi^{+}_{\xi^{n+j}c}$ if $n$ is even and 
$\psi^{\pm}_{\xi^jc}=\psi^{\pm}_{\xi^{n+j}c}$ if $n$ is odd. 
So these $\psi^{\pm}_{\xi^jc}$ give rise to exactly $2n$ distinct automorphisms. 
By Corollary \ref{cor:automorphism-groups}, we have
$$D_{n}\cong \Aut(R_2(p))=\{\phi_{\xi}^j, \psi_c\phi_\xi^j\ |\ j=0,1,\ldots,2n-1\}.$$
This provides a realization of the group $D_{n}$.
\item[(3)] Suppose that $l=1$, $k=2n$. Denote $c=\xi c_1$. Without loss of generality, we may assume 
$\xi^n=\imath$.
Noticing $\xi^{2n}=-1$, we have $\prod_{i=1}^{2n}\a_i=c^{2n}$ and
$$(c\xi^{j+i})(c\xi^{j-i})=(\xi^{j}c)^2,\qquad 
\prod_{i=1}^{2n}\a_i=
\begin{cases}
(\xi^jc)^{2n} &\ \text{if\ $j$\ is\ even},  \\ 
-(\xi^j c)^{2n} &\ \text{if\ $j$\ is\ odd}. 
\end{cases}
$$
We can define $\psi^{\pm}_{\xi^jc}$ for $j=0,1,\ldots,4n-1$. More explicitly, we have 
$$\begin{cases}
\psi^{\pm}_{\xi^jc}: t\mapsto \xi^{2j}c^2t^{-1}, \quad \phi(u)=\pm\xi^{j(n+1)}c^{n+1}t^{-n-1}u & \text{if\ $j$\ is\ even},\\\\
\psi^{\pm}_{\xi^jc}: t\mapsto \xi^{2j}c^2t^{-1}, \quad \phi(u)=\pm\imath\xi^{j(n+1)}c^{n+1}t^{-n-1}u & \text{if\ $j$\ is\ odd}.
\end{cases}$$
Denote $\psi_c=\psi^+_{c}$. We can check that 
$\psi^{+}_{\xi^jc}=\psi_c\phi_\xi^{j}$ if $j\equiv 0\mod 4$ or $j\equiv 3\mod 4$ and
$\psi^{-}_{\xi^jc}=\psi_c\phi_\xi^{j}$ if $j\equiv 2\mod 4$ or $j\equiv 1\mod 4$.
Moreover,
$\psi^{-}_{\xi^jc}=\psi^{+}_{\xi^{2n+j}c}$ if $n$ is even and 
$\psi^{\pm}_{\xi^jc}=\psi^{\pm}_{\xi^{2n+j}c}$ if $n$ is odd. 
So these $\psi^{\pm}_{\xi^jc}$ give rise to exactly $4n$ distinct automorphisms. 
Using Corollary \ref{cor:automorphism-groups}, we can write
$$\U_{n}\cong \Aut(R_2(p))=\{\phi_{\xi}^j, \psi_c\phi_\xi^j :  j=0,1,\ldots,4n-1\}.$$
This provides a realization of the group $\U_{n}$.
\end{itemize}
\end{example}

By the above example, we can deduce the following description of the groups $\Dic_n=\G_n, \U_n, \V_{2n}$.
\begin{corollary} Let $n$ be any positive integer. Then
\begin{enumerate}
\item\label{item:description1} $\mathrm{Dic}_n=\mathbb G_n=\langle x,y : y^{2n}=1, x^2=y^n, yxy=x\rangle=\{y^i, xy^i :  i=0,1,\ldots,2n-1\}$, 
\color{red}
% (not sure, depending on (3) of the above example)
\color{black}
\item\label{item:description2} If $n$ is odd, then $\U_n\cong D_n$. If $n$ is even, then $|\U_n|=4n$ and moreover
$$\mathbb U_n =\langle x,y :  y^{2n}=1,x^{2}=1,yxy=xy^n \rangle=\{y^i,xy^i :  i=0,1,\ldots,2n-1\}. $$
% \item[(3)]  If $n$ is even, then $\V_{2n}\cong \U_n$,  and moreover
% $$\V_{2n} =\langle x,y\,|\,y^{2n}=1,x^{4}=1, (xy)^2=(x^{-1}y)^2=1\rangle=\{y^i,xy^i\ |\ i=0,1,\ldots,2n-1\}.$$
% If $n$ is odd, then $\U_n\cong D_n$, 
\end{enumerate}
\end{corollary}

\begin{proof} Assertion \eqref{item:description1} and the second statement of \eqref{item:description2} follow directly from Corollary \ref{cor:automorphism-groups} and the above example. 
To prove the first statement of \eqref{item:description2}, let $x, y$ be the generators of $\U_n$ as described in \eqref{item:description2}.
Then we have $yx=xy^{n-1}$, and hence $y^nx=xy^{n(n-1)}=x$ since $n$ is odd. 
We get $y^n=1$ and it follows that $\U_n\cong D_n$ if $n$ is odd.
%
%For (3), first assume $n$ is even. Let $x_1, y_1$ be the generators of $\U_n$ as described in \eqref{U'}.
%Then we have $y_1^{2n}=1, x_1^4=1, (x_1y_1)^2=1, (x_1^{-1}y_1)^2=(x_1y_1^{n+1})^2=x_1y_1^{n+1}x_1y_1^{n+1}=x_1y_1x_1^2x_1x_1^2y_1=1$, satisfying the generating elations of $\V_{2n}$ in (3). 
%So there is an epimorphism from $\mathbb V_{2n}$ to $\U_n$. By (2), we have $|\mathbb U_n|=4n$. On the other hand, we can easily see that $|\V_{2n}|\leq 4n$, forcing the above mentioned epimorphism to be an isomorphism.
%
%Then we assume that $n$ is odd. To be continued.
\end{proof}

We add to this another 
 \begin{corollary}[\cite{MR3845909}, Corollary 6.4] \label{polycoefficients}
 Let $p(t)=t(t-\alpha_1)\cdots (t-\alpha_{2n})$, where $\alpha_i$ are distinct roots.   Two possible types of automorphisms $\phi\in \Aut(R_2(p))$ of the algebra $R_2(p)$ are the following: 
 \begin{enumerate}
 \item\label{item:cor-cyclic} If $\alpha_{\gamma(i)}=\zeta \alpha_i$ for some $2n$-th root of unity $\zeta$ and $\gamma\in S_{2n}$, then 
 \begin{equation}
 \phi(t)=\zeta t,\quad \phi(u)=\pm\xi u, 
 \end{equation}
 where we can take $\xi=\zeta^{1/2}=\exp(2\pi\imath/2k)$ with $\zeta$ having order $k$ and  $k|2n$. 
 It follows that $\phi$ has order $2k$. 
 In particular, after a change in indices
\begin{equation}\label{eqn:dihedral-setting-p-in-k-and-n}
 \begin{split}
 p(t) &=
t(t-\alpha_1)(t-\zeta\alpha_1)\cdots (t-\zeta^{k-1}\alpha_1) \cdots (t-\alpha_{2n/k})\cdots (t-\zeta^{k-1}\alpha_{2n/k})\\
&=t(t^k-\alpha_1^k)(t^k-\alpha_2^k)\cdots(t^k-\alpha_{2n/k}^k)    \\
&=\sum_{q=0}^{\frac{2n}{k}}(-1)^qe_q(\alpha_1^k,\dots, \alpha_{2n/k}^k)t^{2n-qk+1}, \\ 
\end{split}
\end{equation}
where $e_q(x_1,x_2,\dots, x_{2n/k})$ is the elementary symmetric polynomial of degree $q$ in $x_1,\dots ,x_{2n/k}$:
\begin{align*}
e_q(x_1,x_2,\dots, x_{2n/k})&=\sum_{1\leq j_1<j_2<\ldots <j_q\leq 2n/k}x_{j_1} x_{j_2} \cdots x_{j_q}.
\end{align*} 
In this case,  $\langle \phi_\xi^+\rangle\cong \Z_{2k}$. 
 
 %  \color{red} $(\phi_\xi^+)^k(u)= \zeta^{nk}\zeta^{k/2}u= \zeta^{k/2}u=-u$\color{black}
 \item\label{cor:group-action-dihedral-maps} 
 If, in addition to the above, 
 there exists $\beta\in S_{2n}$ such that $\alpha_i \alpha_{\beta(i)}=c^2$ for all $i$, then  $\phi_\xi^\pm(t)=\zeta t$ and $\phi_\xi^\pm(u)=\pm\xi u$, and 
 $\psi(t)=c^2t^{-1}$ and 
\begin{equation}\label{pma}
 \psi_c^\pm( u)=\pm t^{-n-1} c^{n+1}u \quad \text{ if }a_1=\prod_{i=1}^{2n}\alpha_i=c^{2n},\tag{a}
 \end{equation}
 or 
\begin{equation}\label{pmb}
 \psi_c^\pm(u)=\pm t^{-n-1}\imath c^{n+1}u\quad \text{ if }a_1=\prod_{i=1}^{2n}\alpha_i=-c^{2n}.\tag{b}
\end{equation} 
In this case
\begin{equation}\label{coefficientsymmetry}
 p(t)=\sum_{r=1}^{2n+1}a_rt^r,  \quad \mbox{ where } \quad  a_k=\pm c^{2n-2k+2}a_{2n+2-k} 
\end{equation}
for $k=1,\dots, 2n+1$.  Here the $\pm$  in \eqnref{coefficientsymmetry} corresponds to the $\pm$ in $a_1=\pm c^{2n}$. 
 \end{enumerate}
 \end{corollary}

%Combining Cox's module decomposition paper and this paper [\color{red}pick a notation and stay consistent in this section; use $\phi$ (rotations of order $r=2n$) and $\psi$ (reflections of order $2$)\color{black}]:  
% 
%Note: $\tau_{\zeta}^+=\phi_{\zeta^j}^+$, $\tau_1^-=\phi_1^-$
%
%Note: $\sigma_c^+ =\psi_c^+$, $\tau_{\zeta}^+ = \phi_{\zeta}^+$, $\tau_1^-=\phi$
%
%Note: $\phi(u)=\tau_1^-(u)= -\zeta^{n+\frac{1}{2}}u$.

%We denote by $\tau_{\zeta}^{\pm}$  and $\sigma_c^{\pm}$ the Lie algebra automorphisms corresponding to the associative algebra automorphisms $\phi_{\zeta}^{\pm}$ and $\psi_{c}^{\pm}$, respectively, in Corollary~\ref{cor:group-actions} (if they exist).  

From \cite{MR966871}, we know that for any
automorphism $\phi$ of the associative algebra $R_2(p)$, one obtains
an automorphism $\tau$ of the Lie algebra $\mathcal{R}_2(p):=\text{Der}(R_2(p))$ through the equation
\begin{equation}\label{eq:Liealgauto}
\tau(f(t)\partial)=\phi(f)(\phi\circ\partial\circ \phi^{-1})\enspace \quad  \mbox{ for all } f \in R_2(p).
\end{equation}
In addition, any Lie algebra automorphism of
$\mathcal{R}_2(p)$ can be obtained from \eqnref{eq:Liealgauto}. Denote by
$\tau_{\zeta}^{\pm}$ and $\sigma_{c}^{\pm}$ the Lie algebra
automorphisms corresponding to the associative algebra automorphisms
$\phi_{\zeta}^{\pm}$ and $\psi_{c}^{\pm}$ in \thmref{cor:group-actions}
\eqref{item:cor-cyclic-thm} and \eqref{cor:group-action-dihedral-maps-thm} respectively (if they indeed exist). For convenience,
denote
$\tau_{\zeta} :=\tau_{\zeta}^{+}$  and  $\sigma_{c}:=\sigma_{c}^{+}.$

 \section{Character Tables}
\subsection{Character Table and Irreducible Finite Dimensional Representations of $\mathbb U_n$}
 
  Recall 
 \begin{equation}
 \mathbb U_n=\langle x,y\,|\,x^2,y^{2n},xyxy^{n+1} \rangle.
 \end{equation}
 
 \begin{prop}  If $n$ is odd, then $\mathbb U_n=D_{n}$ is the dihedral group.

 Suppose $n=2(2s+1)$.   The conjugacy classes of $\mathbb U_{n}$ are 
 \begin{gather*}
 \{1\},\quad \{y^{2s+1}\}, \quad\{y^n\},\quad \{y^{6s+3}\}, \\
  \{y^{2j},y^{-2j}\} \quad 1\leq j<2s+1=n/2,\quad   \{y^{2j+1},y^{n-2j-1}\},\enspace  \{y^{n+2j+1},y^{2n-2j-1}\} \quad 1\leq 2j+1< l=n/2, 
\\  \{x,xy^4,\dots ,xy^{2n-4}\},\quad  \{xy,xy^5,\dots ,xy^{2n-3}\},\quad  \{xy^2,xy^6,\dots ,xy^{2n-2}\},  \\
 \{x^3,xy^7,\dots ,xy^{2n-1}\} 
 \end{gather*}
 for a total of $n+6$ conjugacy classes.
 
 Suppose now $n=4s$.  The conjugacy classes of $\mathbb U_{n}$ are then 
\begin{gather*}
 \{1\}, \quad \{y^{2l}\},   \\
  \{y^{2j},y^{-2j}\} \quad 1\leq j<l,\quad   \{y^{2j+1},y^{n-2j-1}\},\enspace  \{y^{n+2j+1},y^{2n-2j-1}\} \quad 1\leq 2j+1< l, 
\\  \{x,xy^2,\dots, xy^{2n-2}\},\quad  \{xy,xy^3,\dots ,xy^{2n-1}\} 
 \end{gather*}
for a total of $n+3$ conjugacy classes.
 \end{prop}
 \begin{proof}
 First we make the observation for $n$ odd.  
 We have $
yx=xy^{n-1}$ as $x^2=1$, $xyxy^{n+1}=1$ and $y^{2n}=1$.   Hence
$y^{2k+1}x=xy^{(2k+1)(n-1)}=xy^{n-1-2k}$.   Thus if $n=2l+1$ is odd we have $y^nx=y^{2l+1}x=xy^{n-2l-1}=x$.  But then $y^n=1$.   Thus our generators and relations would be 
$$
x^2=1,\quad y^n=1, \quad xyx=y^{-1},
$$
and $\mathbb U_n=D_{n}$ is the dihedral group.

 For the rest of the proof, we assume $n=2l$ is even.  We have $y^kx=xy^{k(n-1)}$ so 
 $$
 xy^{2j}x=y^{2j(n-1)}=y^{-2j}
 $$
 and 
 $$
 xy^{2j+1}x=y^{(2j+1)(n-1)}=y^{n-2j-1}.
 $$
 Thus we have two element conjugacy classes
 $$
 \{y^{2j},y^{2n-2j}\},\quad \{y^{2j+1},y^{n-2j-1}\}
 $$
 provided $y^{2j}\neq y^{2n-2j}$ and $y^{2j+1}\neq y^{n-2j-1}$.   For $1\leq j< n/2 $ one has $y^{2j}\neq y^{2n-2j}$ and we get all of the distinct such conjugacy classes if $1\leq j<n/2$. 
 
 We break this up into two different cases $n=2(2s+1)$ and $n=4s$, where $l=2s+1$ or $l=2s$, respectively. 
 
 \vskip 10pt
 \noindent{\bf Case when $n=2(2s+1)$.} 
  For $y^{2j+1}= y^{n-2j-1}$, we have $4j+2\equiv n\mod 2n$ so that $2(2j+1)=(2k+1)n$ for some $k$.  Hence $n$ must be even so $n=2l$.  Now $2j+1=(2k+1)l$ and $l$ must be odd.  Moreover $1\leq 2j+1=(2k+1)l<2n=4l$.  This implies $2k+1\in \{1,3\}$ or $2j+1=l$ or $2j+1=3l$.    If $n=2l$, then for $2j+1=l$ or $3l$ we get 
  $y^{n-2j-1}=y^{2l-l}=y^l=y^{2j+1}$ and $y^{n-2j-1}=y^{2l-3l}=y^{-l}=y^{3l}=y^{2j+1}$ in the later case as $y^{4l}=y^{2n}=1$.
 
 Consider now the case $n$=2l even with $l=2s+1$ odd and $n>4$.  We have the set of elements 
 \begin{equation}
 \begin{split}
 \{\underline{1}, y, \widehat{y^2},\dots, &\boxed{y^{2j+1}},\dots, \underline{y^l},\dots, \boxed{y^{2l-2j-1}},\dots, \underline{y^{2l}},\dots, \\ &\overline{y^{n+2r+1}},\dots ,\underline{y^{3l}},\dots, \overline{y^{2n-2r-1}},\dots,  \widehat{y^{2n-2}},y^{2n-1}\} \\ 
 \end{split}
\end{equation}
that breaks up into  conjugacy classes with one element in them (the underlined elements above): 
 $$
\{1\},\quad \{y^{l}\},\quad \{y^{2l}=y^n\},\quad \{y^{3l}\},
 $$
 the conjugacy classes with even exponents $\{y^{2j},y^{2n-2j}\}$ with $1\leq j<l=n/2$ (the hatted elements above)
 and the other conjugacy classes with $y$ having an odd positive exponent are of the form
 $$
  \{y^{2j+1},y^{n-2j-1}\} \quad \mbox{ for } 1\leq 2j+1< l, 
 $$
 and 
 $$
  \{y^{n+2j+1},y^{2n-2j-1}\} \quad \mbox{ for } 1\leq 2j+1< l 
 $$
(the boxed and over lined elements above).
There are $l-1=(n/2)-1$ of the latter two types of conjugacy classes.

Using the identities $x^2=y^{2n}=1=xyxy^{n+1}$ we get $yx=xy^{n-1}$ and $y^rx=xy^{r(n-1)}$.  
Hence
 \begin{align*}
 y^rxy^ky^{-r}=xy^{k+r(n-1)-l}=xy^{k+r(n-2)}
 \end{align*}
 and for $r=2q+1$  we get
 \begin{align*}
 y^rxy^ky^{-r}=xy^{k+(2q+1)(n-2)}=xy^{k-4q+n-2}.
 \end{align*}
Consequently recalling $n=2(2a+1)$ we get $y^rxy^ky^{-r}=xy^{k+4(a-q)}$.
Moreover
 \begin{align*}
x (xy^k)x=y^{k}x=xy^{k(n-1)}=xy^{k(4a+1)}=xy^{k+4ak}.
 \end{align*}
 This implies that the set of elements
 $$
 \{x,xy,xy^2,\dots , xy^{2n-1}\}
 $$
 breaks up into four distinct conjugacy classes 
 $$
 \{x,xy^4,\dots xy^{2n-4}\},\quad  \{xy,xy^5,\dots xy^{2n-3}\},\quad  \{xy^2,xy^6,\dots xy^{2n-2}\}, 
 $$
 and 
 $$
  \{x^3,xy^7,\dots xy^{2n-1}\}.
 $$
 Thus in the case of $n=2(2l+1)$ there are exactly $n+6$ conjugacy classes. 
 
  \vskip 10pt
 \noindent{\bf Case when $n=4s$.} 
 Now consider the case of $l=2s$ so $n=4s$. 
We can divide  
$$
 \{\underline{1}, y,\widehat{y^2},\dots,\boxed{y^{2j+1}},\dots, \boxed{y^{2l-2j-1}},\dots, \underline{y^{2l}},\dots \overline{y^{n+2r+1}},\dots  , \overline{y^{2n-2r-1}},\dots,  \widehat{y^{2n-2}},y^{2n-1}\}
 $$
 into  conjugacy classes with one element in them (the underlined elements above): 
 $$
\{1\},\quad \{y^{2l}=y^n\}, 
 $$
 the conjugacy classes with even exponents $\{y^{2j},y^{2n-2j}\}$ with $1\leq j<l=n/2$ (the hatted elements above)
 and the other conjugacy classes with $y$ having an odd positive exponent are of the form
 $$
  \{y^{2j+1},y^{n-2j-1}\} \quad \mbox{ for } 1\leq 2j+1< l, 
 $$
 and 
 $$
  \{y^{n+2j+1},y^{2n-2j-1}\} \quad \mbox{ for }  1\leq 2j+1< l 
 $$
(the boxed and over lined elements above).
There are $2(l/2)=n/2$ of the latter two types of conjugacy classes. 

We lastly want to break 
 $$
 \{x,xy,xy^2,\dots , xy^{2n-1}\}
 $$
into distinct conjugacy classes.  
 For this we observe 
 \begin{equation}\label{conjugationbyy}
 y^rxy^ky^{-r}=xy^{k+r(n-2)}=xy^{k+4rs-2r}
 \end{equation}
 and 
\begin{align*}
x (xy^k)x=y^{k}x=xy^{k(n-1)}=xy^{k(4s-1)}=xy^{4ks-k}=y^k(xy^k)y^{-k}, 
 \end{align*}
 so we need only consider \eqnref{conjugationbyy}.   When $r=2q+1$ we get
 $$
  y^rxy^ky^{-r}=xy^{k+r(n-2)}=xy^{k+(2q+1)(n-2)} =xy^{k+2qn+n-2-4q}=xy^{k+4(s-q)-2}
  $$
  and for $r=2p$ we get
  $$
  y^rxy^ky^{-r}=xy^{k+r(n-2)}=xy^{k+2p(n-2)} =xy^{k-4p}.
  $$
  Since $p$ and $q$ are arbitrary integers we get two conjugacy classes:
  $$
  \{x,xy^2,\dots, xy^{2n-2}\},\quad \{xy,xy^3,\dots, xy^{2n-1}\}.
  $$
Thus in the case of $n=4$ we get exactly $n+3$ conjugacy classes. 
  \end{proof}

\subsubsection{Representations of $\mathbb U_n$, $n$ even} 
\begin{prop}
When $n=2(2l+1)$ the irreducible representations of $\mathbb U_n$ are the eight one dimensional irreducible representations $\rho_1$ (the trivial representation), 
\begin{gather*}
\rho_2(x)=1,\enspace \rho_2(y)=-1,\quad \rho_3(x)=-1,\enspace\rho_3(y)=-1,\quad \rho_4(x)=-1,\enspace\rho_4(y)=1,  \\
\rho_5(x)=1,\enspace \rho_5(y)=\imath,\quad \rho_6(x)=1,\enspace \rho_6(y)=-\imath,\quad \rho_7(x)=-1,\\
\rho_7(y)=-\imath,\quad \rho_8(x)=-1,\enspace\rho_8(y)=\imath,
\end{gather*}
and the two dimensional irreducible representations for $\zeta=\exp(2\pi\imath /2n)$, 
\begin{equation}\label{2dimlrep1}
R_h(x)=\begin{pmatrix} 0 & 1 \\ 1 & 0 \end{pmatrix}
,\quad R_h (y)=\begin{pmatrix} \zeta^h & 0 \\ 0 & (-1)^h\zeta^{-h} \end{pmatrix}
\end{equation}
 for $h$ odd (excluding $h=n/2$ if $n=2(2l+1)$ with $1\leq h<n$ or $n<h<2n$ and $h$ even with $2\leq h<2n$, $h\neq n$. 
  Thus there are $8+(n/2)-1+(n/2)-1=n+6$ irreducible representations.  

When $n=4l$ the irreducible representations of $\mathbb U_n$ are the four one dimensional irreducible representations $\rho_1$ (the trivial representation): 
\begin{gather*}
\rho_2(x)=1,\enspace \rho_2(y)=-1,\quad \rho_3(x)=-1,\enspace\rho_3(y)=-1,\quad \rho_4(x)=-1,\enspace\rho_4(y)=1,  
\end{gather*}
and the two dimensional irreducible representations for $\zeta=\exp(2\pi\imath /2n)$ given by \eqnref{2dimlrep1}.  
In this case, there are $4+(n/2)+(n/2)-1=n+3$ irreducible representations.  
\end{prop}
Note in the case $n=4l$ above we have the sum of the squares of the dimensions of the irreducible representations with their multiplicities is $4+4(n-1)=4n=|\mathbb U_n|$ as it should be and also the number of conjugacy classes $n+3$ is equal to the number of inequivalent irreducible representations.   In the case $n=2(2l+1)$ we have $n+6=8+(n-2)$ conjugacy classes and $8+(n-2)$ irreducible representations. Then we can check  the sum of the squares of the dimensions of the irreducible representations with their multiplicities  $8+4(n-2)=4n$. 
\begin{proof}  For a 1-dimensional representation $\rho$ of $\mathbb U_k$ we must check the defining relations for $\mathbb U_k$:
 $$
 \rho(x)\rho(y)\rho(x)\rho(y)^{n+1}=1=\rho(x)^2=\rho(y)^{2n}
 $$
where $\rho(x),\rho(y)\in\mathbb C^\times$.  Thus $\rho(x)=\pm 1$ and from the first equation and last equations we get $\rho(y)^{n+2}=1=\rho(y)^{2n}$ so $\rho(y)^{2n+4}=(\rho(y)^{n+2})^2=1$.   Hence $\rho(y)^4=1$ and $\rho(y)\in\{\pm 1,\pm \imath\}$.  When $n=4l$ we have $\rho(y)^n=\rho(y)^{4l}=1=\rho(y)^{n+2}$ so that $\rho(y)^2=1$ and $\rho(y)=\pm 1$.

 For $0\leq h<2n$ we have representations $R_h:\mathbb U_n\to \text{GL}_2(\mathbb C)$ as above since
 $$
 R_h(x)^2=\begin{pmatrix} 0 & 1 \\ 1 & 0 \end{pmatrix}^2=I_2=R_h (y)^{2n}=\begin{pmatrix} \zeta^h & 0 \\ 0 & (-1)^h\zeta^{-h} \end{pmatrix}^{2n},
 $$
 and 
\begin{align*}
  R_h(x)R_h(y)R_h(x)R_h(y)^{n+1}&=\begin{pmatrix} 0 & 1 \\ 1 & 0 \end{pmatrix}\begin{pmatrix} \zeta^h & 0 \\ 0 & (-1)^h\zeta^{-h} \end{pmatrix}\begin{pmatrix} 0 & 1 \\ 1 & 0 \end{pmatrix}\begin{pmatrix} \zeta^h & 0 \\ 0 & (-1)^h\zeta^{-h} \end{pmatrix}^{n+1}  \\
  &= \begin{pmatrix}  0 & (-1)^h\zeta^{-h} \\\zeta^h & 0  \end{pmatrix} \begin{pmatrix} 0 & 1 \\ 1 & 0 \end{pmatrix}\begin{pmatrix} \zeta^{h(n+1)} & 0 \\ 0 & (-1)^{h(n+1)}\zeta^{-h(n+1)} \end{pmatrix} \\
  &= \begin{pmatrix}  0 & (-1)^h\zeta^{-h} \\\zeta^h & 0  \end{pmatrix} \begin{pmatrix} 0 & 1 \\ 1 & 0 \end{pmatrix}\begin{pmatrix} (-1)^h\zeta^{h} & 0 \\ 0 & \zeta^{-h} \end{pmatrix}  \\
  &= \begin{pmatrix}  0 & (-1)^h\zeta^{-h} \\\zeta^h & 0  \end{pmatrix}  \begin{pmatrix}  0 & \zeta^{-h} \\ (-1)^h\zeta^{h} & 0\end{pmatrix}  \\
  &=I_2
 \end{align*}
 as $n$ is even.  Consequently $R_h$ is a representation. 
 
 For $1\leq h<n$ odd, the representations $R_h$ and $R_{n-h}$ are isomorphic with change of basis matrix 
 $$
 T=\begin{pmatrix} 0 & -1 \\ -1 & 0\end{pmatrix}
 $$
 as 
 \begin{align}\label{equivalent}
 \begin{pmatrix} 0 & -1 \\ -1 & 0\end{pmatrix}\begin{pmatrix} \zeta^h & 0 \\ 0 & (-1)^h\zeta^{-h} \end{pmatrix} \begin{pmatrix} 0 & -1 \\ -1 & 0\end{pmatrix}&=\begin{pmatrix} (-1)^h\zeta^{-h} & 0 \\ 0 & \zeta^{h} \end{pmatrix} =\begin{pmatrix} \zeta^{n-h} & 0 \\ 0 & \zeta^{-n+h} \end{pmatrix}.  
 \end{align}
  Similarly $R_g$ and $R_{3n-g}$ are equivalent for odd $g$ with $n<g<2n$.  Note that when $h$ is even \eqnref{equivalent} gives us $R_h$ and $R_{2n-h}$ are equivalent. 
 
 Let us show that $R_h$ and $R_g$ are not equivalent for $1\leq h <n<g<2n$ for odd $g$ and $h$.  Indeed suppose that such an isomorphism exits.  Then there exists a change of basis matrix 
 $$
 \begin{pmatrix} a & b \\ c & d \end{pmatrix} 
 $$
 implementing this isomorphism.  
 
 Simplifying for $h$ odd
 $$
 \begin{pmatrix} a & b \\ c & d \end{pmatrix} 
\begin{pmatrix} 0 & 1 \\ 1 & 0 \end{pmatrix}
=\begin{pmatrix} 0 & 1 \\ 1 & 0 \end{pmatrix}
 \begin{pmatrix} a & b \\ c & d \end{pmatrix}  \quad\text{ and }\quad  
 \begin{pmatrix} a & b \\ c & d \end{pmatrix} \begin{pmatrix} \zeta^h & 0 \\ 0 & -\zeta^{-h} \end{pmatrix} =\begin{pmatrix} \zeta^g & 0 \\ 0 & -\zeta^{-g} \end{pmatrix}
 \begin{pmatrix} a & b \\ c & d \end{pmatrix} $$
 gives us 
 $$
 \begin{pmatrix} a & b \\ c & d \end{pmatrix} =\begin{pmatrix} a & b \\ b& a \end{pmatrix}  \quad\text{ and }\quad   \begin{pmatrix} \zeta^h a &  -\zeta^{-h} b \\ \zeta^h b&  -\zeta^{-h} a \end{pmatrix} =
 \begin{pmatrix} \zeta^g a & \zeta^g b \\ -\zeta^{-g} b& -\zeta^{-g} a \end{pmatrix}. 
 $$

 Thus $\zeta^{h-g}a=a$ and $\zeta^{h+g}b=-b$.   If $a\neq 0$, then $h\equiv g\mod 2n$.  If $a=0$, then $h+g\equiv n\mod 2n$ so that $h+g=(2k+1)n$ but $h+g<n+2n=3n$.  Thus $2k+1=1$ and we get $h+g=n$ so that $g=n-h<n$, a contradiction.  
 
 Next we show that these representations are irreducible for $h$ odd when $n=4s$.  Otherwise say $R_h$ fixes a one dimensional space spanned by $(a,b)^t$. 
 Then the only one dimensional subspaces fixed by $R_h(x)$ are spanned by
 $$
 \begin{pmatrix}
 \:\:\:1 \\ 
 \pm 1 
 \end{pmatrix}.
 $$
 On the other hand 
  $$
\begin{pmatrix} \zeta^h & 0 \\ 0 & -\zeta^{-h} \end{pmatrix} \begin{pmatrix}1 \\  1 \end{pmatrix} =\begin{pmatrix} \zeta^h \\ -\zeta^{-h} \end{pmatrix} = \zeta^h \begin{pmatrix}1\\ -\zeta^{-2h} \end{pmatrix}
 $$
 which is an eigenvalue if and only if $\zeta^{2h}=-1$ or $h=n/2$.    Now $h$ odd means that $n=2(2s+1)$. If $n=4s$, then there is no invariant vector. 
Similarly
 $$
\begin{pmatrix} \zeta^h & 0 \\ 0 & -\zeta^{-h} \end{pmatrix} \begin{pmatrix}1 \\  -1 \end{pmatrix} =\begin{pmatrix} \zeta^h \\ \zeta^{-h} \end{pmatrix}= \zeta^h\begin{pmatrix}1  \\ \zeta^{-2h} \end{pmatrix}
 $$
leads also to $n=2(2s+1)$.   As a consequent we have $(n/2)-1$ irreducible representations $R_h$ , $h$ odd, if $n=4s$ and $n/2$ if $n=4s$. 
 
  Note that when $a$ is even \eqnref{equivalent} gives us $R_a$ and $R_{2n-a}$ are equivalent.  Moreover $R_n$ is reducible as
  $$
  R_n(y)=\begin{pmatrix} \zeta^n & 0 \\ 0 & (-1)^n\zeta^{-n} \end{pmatrix}=\begin{pmatrix}-1 & 0 \\ 0 &-1 \end{pmatrix}.
  $$
 Simplifying for $1\leq h,g<n$ 
 $$
 \begin{pmatrix} a & b \\ c & d \end{pmatrix} 
\begin{pmatrix} 0 & 1 \\ 1 & 0 \end{pmatrix}
=\begin{pmatrix} 0 & 1 \\ 1 & 0 \end{pmatrix}
 \begin{pmatrix} a & b \\ c & d \end{pmatrix}  \quad\text{ and }\quad  
 \begin{pmatrix} a & b \\ c & d \end{pmatrix} \begin{pmatrix} \zeta^{2h} & 0 \\ 0 & \zeta^{-2h} \end{pmatrix} =\begin{pmatrix} \zeta^{2g} & 0 \\ 0 & \zeta^{-2g} \end{pmatrix}
 \begin{pmatrix} a & b \\ c & d \end{pmatrix} $$
 give us 
 $$
 \begin{pmatrix} a & b \\ c & d \end{pmatrix} =\begin{pmatrix} a & b \\ b& a \end{pmatrix}  \quad\text{ and }\quad   \begin{pmatrix} \zeta^{2h} a &  \zeta^{-2h} b \\ \zeta^{2h} b&  \zeta^{-2h} a \end{pmatrix} =
 \begin{pmatrix} \zeta^{2g} a & \zeta^{2g} b \\ \zeta^{-2g} b& \zeta^{-2g} a \end{pmatrix}. 
 $$
Hence either $2h=2g\mod 2n$ or $2h+2g\equiv 0\mod 2n$.   In the first case $h-g=kn$ for some $k$. But $1\leq h,g\leq n$ leads to $h=g$.  If $h+g=nk$ for some $k$, then since $1\leq g,h<n$ one gets $k=1$ and $g=n-h$ or $2g=2n-2h$ which is the equivalence given above from \eqnref{equivalent}.  

Let us check that $R_{2h}$ are irreducible for $2\leq 2h<n$.  Otherwise $R_{2h}$ fixes a one dimensional subspace spanned by either $(1,1)^t$ or $(1,-1)^t$. We calculate  
$$
\begin{pmatrix} \zeta^{2h} & 0 \\ 0 & \zeta^{-2h} \end{pmatrix} \begin{pmatrix} 1 \\ 1 \end{pmatrix}=\begin{pmatrix} \zeta^{2h} \\   \zeta^{-2h} \end{pmatrix}=\zeta^{2h} \begin{pmatrix} 1 \\   \zeta^{-4h} \end{pmatrix}
$$
so that one would need $4h= 2n$ or $2h=n$.  But $2h<n$. 

Similarly
$$
\begin{pmatrix} \zeta^{2h} & 0 \\ 0 & \zeta^{-2h} \end{pmatrix} \begin{pmatrix} 1 \\ -1 \end{pmatrix}=\begin{pmatrix} \zeta^{2h} \\   -\zeta^{-2h} \end{pmatrix}=\zeta^{2h} \begin{pmatrix} 1 \\  - \zeta^{-4h} \end{pmatrix}
$$
implies $4h=2n$ again a contradiction.   Consequently there are $(n/2)-1$ irreducible representations of the form $R_{2h}$ with $2\leq 2h<n$. 

 Observe that $R_h$, $h$ odd, can not be equivalent to $R_g$, $g$ even as otherwise there is a change of basis matrix $A$ such that  $AR_h(a)A^{-1}=R_g(a)$ for all $a\in\mathbb U_n$ which means in particular that $-1=\det R_h(y)=\det R_g(y)=1$.
\end{proof}
 \subsection{Character Table and Irreducible Finite Dimensional Representations of $\text{Dic}_n$}
The dicyclic group is defined to be 
$$
 \mathbb G_n=\text{Dic}_n:=\langle y,x : y^{2n}=1,\enspace x^2=y^n,\enspace yxy=x\rangle.
$$

\begin{prop}[\cite{MR1119304}]\label{dicn}
 The conjugacy classes of $\mathbb G_{n}$ are 
 \begin{enumerate}
\item\label{item-conj1} $\{1\}, \{y^n\}, \{y^i, y^{2n-i}\} \quad \text{where\ } 1\leq i\leq n-1$, 
\item\label{item-conj2} $\{x, xy^2,\ldots,x^{2n-2}\}, \{xy, xy^3,\ldots, xy^{2n-1}\}.$
\end{enumerate}
\end{prop}

\begin{proof} From $yxy=x$, we have $x^{-1}y^ix=y^{-i}=x^{2n-i}$. So we have the 
conjugacy classes in \eqref{item-conj1}. Then we calculate 
$$x^{-1}(xy^i)x=y^ix=xy^{2n-i}\quad\text{and}\quad y(xy^i)y^{-1}=yxy^{i-1}=xy^{i-2}.$$
So we get two conjugacy classes in \eqref{item-conj2}.
\end{proof}

\subsubsection{Irreducible representations}
%{\color{red} \textbf{Ben Cox} will look up the whether or not the following was known to Coxeter. }
This section was certainly known to Coxeter in \cite{MR1119304}.  
\color{black}
The dicyclic group $\text{Dic}_n$ is as double cover of the dihedral group: 
\begin{equation}\label{Dick}
\xymatrix{
 \{\pm 1\} \ar@{^{(}->}[r]^{y^n} 		&	\text{Dic}_n  \ar@{>>}[r] 	& D_n.  }
\end{equation}
\subsection{When $n$ is even}  We have $(n-2)/2$ inequivalent irreducible representations $\sigma_k:\text{Dic}_n\to \text{GL}_2(\mathbb C)$ given by 
\begin{align}
\sigma_k(a)=\begin{pmatrix} e^{\pi\imath k/n} & 0 \\ 0 & e^{- \pi \imath  k/n}\end{pmatrix} ,\quad 
\sigma_k(x)= \begin{pmatrix} 0 & 1 \\ 1 & 0\end{pmatrix}
\end{align}
where $k=1,2,\dots, (n-2)/2$.

By \propref{dicn} one knows that there are $n+3$ conjugacy classes and thus there are $n+3$ irreducible representations (see \cite[Theorem 7, page 19]{MR0450380}). 

There are $n/2$  irreducible representations of degree $2$ of the form
\begin{align}
\tau_k(a)=\begin{pmatrix} e^{\pi\imath k/n} & 0 \\ 0 & e^{- \pi \imath  k/n}\end{pmatrix} ,\quad 
\tau_k(x)= \begin{pmatrix} 0 & -1 \\ 1 & 0\end{pmatrix}
\end{align}
where $1\leq k\leq n$ is odd.

In addition there is the trivial representation $\omega_0$ and the representations
\begin{align*}
\omega_1(y)&=1,\qquad \omega_1(x)=-1, \\
\omega_2(y)&=-1,\qquad \omega_2(x)=1, \\
\omega_3(y)&=-1,\qquad \omega_3(x)=-1. 
\end{align*}

\subsection{When $n$ is odd} 
Again we describe $n+3$ irreducible representations over the complex numbers. 

We have $(n-1)/2$ inequivalent irreducible representations $\sigma_k:\text{Dic}_n\to \text{GL}_2(\mathbb C)$ given by 
\begin{align}
\sigma_k(a)=\begin{pmatrix} e^{\pi\imath k/n} & 0 \\ 0 & e^{- \pi \imath  k/n}\end{pmatrix} ,\quad 
\sigma_k(x)= \begin{pmatrix} 0 & 1 \\ 1 & 0\end{pmatrix}
\end{align}
where $k=1,2,\dots, (n-1)/2$.

There are $(n-1)/2$ other irreducible representations of degree $2$ of the form
\begin{align}
\gamma_k(a)=\begin{pmatrix} e^{\pi\imath k/n} & 0 \\ 0 & e^{- \pi \imath  k/n}\end{pmatrix} ,\quad 
\gamma_k(x)= \begin{pmatrix} 0 & -1 \\ 1 & 0\end{pmatrix}
\end{align}
where $k=1,2,3,\dots, n-2$.

In addition there is the trivial representation $\omega_0$ and the representations
\begin{align*}
\omega_1(y)&=-1,\quad \omega_1(x)=1, \\
\omega_2(y)&=-1,\quad \omega_1(x)=\imath, \\
\omega_3(y)&=-1,\quad \omega_2(x)=-\imath.
\end{align*}

\section{The decomposition of the space of K\"ahler differentials modulo exact forms for $D_{2k}$}\label{section:main-theorem}
Let $r=2n$, $R=R_2(p)$ and 
let $G:=\Aut(R)$ be the groups in Corollary~\ref{cor:automorphism-groups}. %, and consider $G$ acting on the K\"ahler differential $\Omega_R/dR$. 
For $\phi\in G$ and $\overline{rds} \in \Omega_R/dR$,  %, which has a finite dimensional basis, 
the action of $G$ on the K\"ahler differential is given by: 
\begin{equation}
\phi(\overline{rds}) = \overline{\phi(r)d\phi(s)}. \label{groupaction}
\end{equation}

First we include the following lemma for completeness. 

\begin{lemma}[\cite{MR3845909}, Lemma 7.1]
If $n$ is even, the character table is given as follows: 
\footnotesize
$$
M=\bordermatrix{\text{}&\rho_1&\rho_2&\rho_3 & \rho_4 & \chi_1 & \ldots &\chi_h & \ldots & \chi_{(n/2)-1} \cr
              1 &1 &  1  & 1 & 1 & 2 & \ldots & 2& \ldots & 2 \cr
               \psi & 1  &  -1& 1 & -1 & 0 &\ldots & 0 & \ldots & 0 \cr
                \psi \phi& 1 & -1 & -1& 1 & 0 &\ldots & 0 & \ldots & 0 \cr
                \phi & 1  &   1      &-1 & -1  & 2 \cos\left( \frac{2\pi}{ n}\right) &\ldots & 2\cos\left(\frac{2\pi h}{n}\right) & \ldots & 2\cos\left(\frac{2\pi \left((n/2)-1\right)}{n}\right)    \cr 
                \vdots & \vdots &  \vdots      &\vdots & \vdots  & \vdots & \ddots & \vdots & \ddots & \vdots  \cr 
                \phi^k& 1  &   1      &(-1)^k & (-1)^k   & 2 \cos\left( \frac{2\pi k}{ n}\right) &\ldots & 2\cos\left(\frac{2\pi hk}{n} \right) & \ldots & 2\cos\left(\frac{2\pi k \left((n/2)-1\right)}{n}\right)    \cr 
                \vdots  & \vdots  &   \vdots        &\vdots  & \vdots & \vdots & \ddots & \vdots & \ddots & \vdots \cr 
                \phi^{n/2} & 1  &   1      &(-1)^{n/2} & (-1)^{n/2}   & -2 &\ldots & 2(-1)^h& \ldots & 2(-1)^{(n/2)-1} \cr 
}.$$
So we have 
\begin{align*}
M^{-1} &= M^t \begin{pmatrix}
       \frac{1}{2n} &0 &  0  & 0 &    \ldots & 0& \ldots & \ldots & 0 \\
       0 &  \frac{n/2}{2n}& 0 & 0 &\ldots  & 0 & \ldots &  \ldots &0\\
       0& 0 &  \frac{n/2}{2n} & 0 &\ldots & 0 & \ldots &  \ldots & 0 \\
       0 & 0 & 0 &\frac{2}{2n} &\ldots & 0 & \ldots &  \ldots &0    \\ 
       \vdots & \vdots &\vdots & \vdots  & \ddots & \vdots  & \ddots & \vdots & \vdots  \cr 
       0 & 0 & 0 & 0 &\ldots &\frac{2}{2n}& \ldots & \vdots & 0  \\
       \vdots & \vdots &\vdots  & \vdots & \ddots & \vdots & \ddots & \vdots & \vdots \\
       0& 0 & 0 & 0 & \ldots & \vdots & \ldots &  \frac{2}{2n} & 0    \\
       0  & 0 & 0 & 0 & \ldots & 0& \ldots &  0& \frac{1}{2n} \\ 
\end{pmatrix}\\
&=\begin{pmatrix}
               \frac{1}{2n}  & \frac{1}{4} &  \frac{1}{4}  & \frac{1}{n} &    \ldots & \frac{1}{n}& \ldots &  \frac{1}{2n} \\
                \frac{1}{2n}   &  -\frac{1}{4}&- \frac{1}{4} &  \frac{1}{n} &  \ldots  & \frac{1}{n} & \ldots &  \frac{1}{2n} \\
                   \frac{1}{2n}  & \frac{1}{4} & -\frac{1}{4}& - \frac{1}{n}  &\ldots &  \frac{(-1)^k }{n}& \ldots &   \frac{(-1)^{n/2}}{2n}   \\
                \frac{1}{2n}  &   -\frac{1}{4}      & \frac{1}{4} & - \frac{1}{n} & \ldots  & \frac{(-1)^k}{n} & \ldots &\frac{(-1)^{n/2} }{2n}   \\ 
                \frac{1}{n}   &   0 & 0 &  \frac{2}{n} \cos\left( 2\pi /n\right)  &\ldots &  \frac{2}{n}\cos\left(2\pi h/n\right) & \ldots & -\frac{1}{n}   \\ 
                  \vdots &  \vdots      &\vdots & \vdots  & \ddots & \vdots & \ddots & \vdots\\
                   \frac{1}{n}   & 0 &0      & \frac{2}{ n}\cos\left(2\pi h/n\right)  &\ldots &  \frac{2}{n}\cos\left(2\pi hk /n\right) & \ldots & \frac{(-1)^h }{n} \\
                 \vdots  &   \vdots        &\vdots  & \vdots & \ddots &\vdots & \ddots & \vdots \\
                 \frac{1}{n}   &   0      & 0 &  \frac{2}{n}\cos\left(2\pi  \left((n/2)-1\right)/n\right)   &\ldots & \frac{2}{n}\cos\left(2\pi k \left((n/2)-1\right)/n\right)  &\ldots   & \frac{(-1)^{(n/2)-1}}{n} \\ 
\end{pmatrix}.  \\
\end{align*}

\end{lemma}
\begin{lemma}
If $n$ is odd, the character table is given by the matrix 
\footnotesize
$$
M_{\mbox{odd}}=\bordermatrix{\text{}&\rho_1&\rho_2& \chi_1 & \ldots &\chi_h & \ldots & \chi_{(n-1)/2} \cr
              1 &1 &   1 & 2 & \ldots & 2& \ldots & 2 \cr
               \psi & 1  &  -1&  0 &\ldots & 0 & \ldots & 0 \cr
                \phi & 1 & 1 &  2 \cos\left(\frac{2\pi}{n} \right) &\ldots & 2\cos\left( \frac{2\pi h}{n}\right) & \ldots & 2\cos\left(\frac{2\pi \left( \frac{n-1}{2}\right)}{n} \right)    \cr 
                \vdots & \vdots & \vdots  & \vdots & \ddots & \vdots & \ddots & \vdots  \cr 
                \phi^k& 1 & 1 & 2 \cos\left( \frac{2\pi k} {n}\right) &\ldots & 2\cos\left( \frac{2\pi h k}{n} \right) & \ldots & 2\cos\left( \frac{2\pi k \left(\frac{n-1}{2}\right)}{n}\right)    \cr 
                \vdots  & \vdots  & \vdots & \vdots & \ddots & \vdots & \ddots & \vdots \cr 
                \phi^{(n-1)/2} & 1 & 1 & 2\cos\left(\frac{2\pi\frac{n-1}{2}}{n}\right) &\ldots & 2\cos\left(\frac{2\pi h\frac{n-1}{2}}{n}\right) & \ldots & 2\cos\left(\frac{2\pi \left(\frac{n-1}{2}\right)^2}{n} \right) \cr 
}.$$
So  
\begin{align*}
M_{\text{odd}}^{-1} &= M_{\text{odd}}^t \begin{pmatrix}
              \frac{1}{2n} &0 &  0  & 0 &    \ldots & 0& \ldots & \ldots & 0 \\
               0 & \frac{n}{2n}& 0 & 0 &\ldots  & 0 & \ldots &  \ldots &0\\
               0 & 0 & \frac{2}{2n} & 0 &\ldots & 0 & \ldots &  \ldots & 0 \\
               0 & 0 & 0 &\frac{2}{2n} &\ldots & 0 & \ldots &  \ldots & 0 \\ 
                  \vdots & \vdots &\vdots & \vdots  & \ddots &  \vdots & \ddots & \vdots & \vdots  \cr 
                  0 & 0 & 0 & 0 &\ldots &\frac{2}{2n}& \ldots & \ldots &0  \\
                 \vdots  &   \vdots &\vdots  & \vdots & \ddots &\vdots & \ddots &\vdots & \vdots \\
                0 & 0 & 0 & 0 & \ldots &\ldots &\ldots &  \frac{2}{2n} & 0    \\
                0 & 0 & 0 &  0 &  \ldots & 0& \ldots &  0& \frac{2}{2n} \\ 
\end{pmatrix}\\
&= 
\begin{pmatrix}
   \frac{1}{2n} & \frac{n}{2n} & \frac{2}{2n} &    \ldots & \frac{2}{2n}& \ldots &  \frac{2}{2n} \\
   \frac{1}{2n} & -\frac{n}{2n}&  \frac{2}{2n} &  \ldots  & \frac{2}{2n} & \ldots &  \frac{2}{2n} \\
   \frac{2}{2n} & 0 & \frac{2}{2n} \cos\left( \frac{2\pi}{n}\right)  &\ldots &  \frac{2}{2n}\cos\left( \frac{2\pi k}{n} \right) & \ldots & \frac{2}{2n} \cos\left( \frac{2\pi \frac{n-1}{2}}{n}\right)  \\ 
   \vdots & \vdots & \vdots  & \ddots &  & \vdots & \vdots\\
   \frac{2}{2n} & 0 & \frac{2}{ 2n}\cos\left(\frac{2\pi h}{n} \right)  &\ldots &  \frac{2}{2n}\cos\left(\frac{2\pi h k}{n} \right) & \ldots & \frac{2}{2n} \cos\left(\frac{2\pi h\frac{n-1}{2}}{n} \right)  \\
   \vdots  & \vdots & \vdots & \vdots & & \ddots & \vdots \\
   \frac{2}{2n} & 0 & \frac{2}{2n}\cos\left( \frac{2\pi  \left(\frac{n-1}{2}\right)}{n}\right)   &\ldots & \frac{2}{2n}\cos\left( \frac{2\pi  k \left(\frac{n-1}{2}\right)}{n}\right)  &\ldots   & \frac{2}{2n}\cos\left( \frac{2\pi  \left(\frac{n-1}{2}\right)^2}{n}\right) \\ 
\end{pmatrix}.  \\
\end{align*} 
\end{lemma}

\begin{proof}
The matrix $M_{odd}$ is the character table for $D_n$ when $n$ is odd. 
From page $38$ in \cite{MR0450380}, we obtain the character table for $n$ odd: 
\[ 
\vline 
\begin{array}{c|c|c}
\hline
 & (\phi_{\zeta}^+)^k   &  \psi_c^+ (\phi_{\zeta}^+)^k  \\  \hline 
\rho_1 & 1 & \:\:\: 1   \\    \hline 
\rho_2 &  1 &  -1 \\    \hline 
\chi_h &  2\cos\left(2\pi hk/n\right)  & \:\:\: 0  \\    \hline 
\end{array}
\vline 
\]  
where $1\leq h\leq (n-1)/2$, $\psi_c^+$ is a reflection, and $\phi_{\zeta}^+$ is a rotation. 

Letting $\Xi$ denote the set of conjugacy classes of the group $D_n$, we get the inverse matrix for $M_{odd}$ from the orthogonality of the characters of the irreducible representations. That is, one uses 
\begin{equation}\label{eqn:orthongalityofchar-n-odd}
\sum_{\{g\}\in\Xi }\frac{|\{g\}|}{|D_n|}\chi_\pi(g)\overline{\chi_\rho(g)}=\begin{cases} 1 
& \text{ for }\pi\cong \rho, \\
0 & \text{ otherwise},\end{cases}
\end{equation}
for any two irreducible representations $\pi$ and $\rho$ of $D_n$ 
(see page 260 in \cite{MR1695775}).
 
   The distinct conjugacy classes of $D_n$ (for $n$ odd) are: 
\[ 
\{\Inoindex \}, 
\{ (\phi_{\zeta}^+),  (\phi_{\zeta}^+)^{-1}\},    \ldots, 
\{ (\phi_{\zeta}^+)^i,  (\phi_{\zeta}^+)^{-i}\}, \ldots, 
\{ (\phi_{\zeta}^+)^{\widetilde{m}},  (\phi_{\zeta}^+)^{-\widetilde{m}}\}, 
\{ \psi_c^+ (\phi_{\zeta}^+)^k: 0\leq k <n\} 
\] 
since odd dihedral groups have trivial center. 
\end{proof}
 
So under an action by $G$, 
we decompose $\Omega_R/dR = Z(\widehat{\mathfrak{g}})$ 
into a direct sum of irreducible representations.  
Our goal in this section is to describe the module structure of $\Omega_R/dR$ into irreducibles under the action by $G$ for a particular $R_2(p)$. 

For completeness, we state the full result when the automorphism group is $D_n$, where $n\in \mathbb{N}$. 
\begin{theorem}\label{theorem-result-001}  
Let 
$p(t)=t(t-\alpha_1)\cdots (t-\alpha_{2n})$, where $\alpha_i$ are pairwise distinct. 
Assume $\sigma_c^{\pm}$ exists in $\Aut(\mathcal{R}_2(p))$ for some nonzero $c\in \mathbb{C}$, $c^{2n}=a_1$ and $k|n$.  
% ($n=2k$ is allowed but too restrictive). 

If $k$ is also even, then under the action of $D_{2k}$ the center decomposes as: 
\begin{equation} 
\Omega_R/dR \cong \mathbb{C} \omega_0 \oplus  \bigoplus_{i=3}^{4} U_i^{\frac{(1-(-1)^k)n}{2k} } \oplus \bigoplus_{h=1}^{k-1} V_{h}^{\oplus \frac{(1-(-1)^h)n}{k}}. 
\end{equation} 
where $U_i$, $i=1,2,3,4$ are the irreducible one dimensional representations for $D_{2k}$ with character $\rho_i$ and $V_h$ are the irreducible 2-dimensional representations for $D_{2k}$ with character $\chi_h$, $1\leq h\leq k-1$.  Note $\mathbb C\omega_0$ and $U_1$ are the trivial representations.

If $k$ is odd, then the center decomposes as: 
\begin{equation}
\Omega_R/dR \cong \mathbb{C}\omega_0 
\oplus \bigoplus_{i=3}^{4} U_i^{\oplus \Upsilon_i(\epsilon_i, \nu_i)} 
\oplus 
\bigoplus_{j=1}^{k-1} V_j^{\oplus \frac{(1-(-1)^j)n}{k}}, 
\end{equation}
where 
\begin{align*} 
\Upsilon_i(\epsilon_i,\nu_i) &= \frac{n}{2k} +\frac{\epsilon_i }{4}   \displaystyle{\sum_{i=n+3}^{2n}}
c^{\frac{n+3-2i}{2}}
  P_{i-n-3,-i} + \frac{\nu_i}{4}\displaystyle{\sum_{i=n+3}^{2n}}
\zeta^{\frac{2n+3-2i}{2}}
c^{\frac{n+3-2i}{2}}
  P_{i-n-3,-i} \\ 
\end{align*} 
with 
\begin{align*} 
\Upsilon_i(\epsilon_i,\nu_i) 
%&= \frac{(1-(-1)^k)n}{2k}(\delta_{i,3}+\delta_{i,4})+\frac{\epsilon_i }{4}   \displaystyle{\sum_{i=n+3}^{2n}}
%c^{n+3-2i}
%  P_{i-n-3,-i} + \frac{\nu_i}{4}\displaystyle{\sum_{i=n+3}^{2n}}
%\zeta^{\frac{2n+3-2i}{2}}
%c^{n+3-2i}
%  P_{i-n-3,-i} \\ 
  &= \frac{(1-(-1)^k)n}{2k}(\delta_{i,3}+\delta_{i,4})+\frac{1-(-1)^n}{4}(\delta_{i,4}-\delta_{i,3})+\frac{1 }{2} (-1)^i  \displaystyle{\sum_{i=n+3}^{2n}}
c^{n+3-2i}
  P_{i-n-3,-i}.  
\end{align*} 

If the automorphism group is $D_n$, where $n$ is odd, then the center decomposes as  
\begin{equation} 
\Omega_R/dR \cong \mathbb{C} \omega_0 \oplus U_1^{\Xi_1} 
\oplus U_2^{\Xi_2} \oplus \bigoplus_{j=1}^{\frac{n-1}{2}} V_{j},  
\end{equation}
 where 
\begin{align*}
\Xi_1 = \frac{1}{2}- \frac{1}{2} \sum_{i=n+3}^{2n} c^{\frac{n+3-2i}{2}}P_{i-n-3,-i} 
\qquad 
\mbox{ and }  
\qquad 
\Xi_2 = \frac{3}{2}+ \frac{1}{2} \sum_{i=n+3}^{2n} c^{\frac{n+3-2i}{2}}P_{i-n-3,-i}, 
\end{align*}
$\mathbb{C} \omega_0$ 
is a $1$-dimensional irreducible representation, $U_i$ are pairwise distinct $1$-dimensional irreducible representations, and 
$V_j$ are pairwise distinct $2$-dimensional irreducible representations. 

When $n=2$, we have $\mathbb{C}\omega_2\oplus \mathbb{C}\omega_3$ and $\mathbb{C}\omega_1\oplus \mathbb{C}\omega_4$ as two $2$-dimensional irreducible representations of $D_2\times \mathbb{Z}_2$. 
 \end{theorem}

%\color{blue}STRATEGY TO PROVE THE THEOREM [\textbf{delete this after completing the proof}]:  For the dihedral group action (Ch 5 in Serre's book), write down the character table under each group action, then use the projection formula in Serre's book or in Cox's module decomp paper to project an arbitrary vector into a $2$-dimensional or $1$-dimensional irreducible component in order to find the basis elements for the irreducible representations.  
%Begin with $n=3$ case, and 
%do a separate case for $n=2$ so $|D_n|=4$.  
%\color{black}

\begin{remark}
If $\sigma_c^{\pm}$ does not exist, then we obtain the automorphism group being $C_{2k}$, and the decomposition is done in Theorem $7.2. (1)$ in \cite{MR3845909}. 
\end{remark}

\begin{corollary}\label{cor:honest-2-dimensional-irreps}
If $\overline{\omega_i}=c^{-\frac{n+3-2i}{2}}\zeta^{-\frac{i}{2}}\omega_i$ for $1\leq i\leq n+2$, then 
$\overline{\omega_i}$ and 
$\overline{\omega_{n+3-i}}$  
span a $2$-dimensional irreducible representation. 
 \end{corollary}
  If $n$ is even, then $1\leq i\leq (n+2)/2$ in Corollary~\ref{cor:honest-2-dimensional-irreps}.  
% CHECK THIS AGAIN! 

 We will now prove Theorem~\ref{theorem-result-001}.  
\begin{proof} % the polynomial p is a product of l cycles.

 When $n$ is even, the proof is given on pages $85$-$90$ in \cite{MR3845909}. 

 For $n$ odd, we have 
$\mathbb{C}\omega_i\oplus \mathbb{C}\omega_{n+3-i}$ forming $2$-dimensional irreducible $D_{n}$-representations for $1\leq i\leq \frac{n+1}{2}$ since similar to the case when $n$ is even, we have 
\begin{equation}
\phi_{\zeta}^+|_{\{\omega_i,\omega_{n+3-i}\}} = 
\left( 
\begin{array}{cc}
\zeta^{\frac{2n+3-2i}{2}} & 0 \\ 
0 & \zeta^{\frac{2i-3}{2}}  \\ 
\end{array}
\right)  
\mbox{ and } 
\psi_c^+ |_{\{\omega_i,\omega_{n+3-i}\}}
= 
\left( 
\begin{array}{cc} 
0&  -c^{-\frac{n+3-2i}{2}} \\ 
-c^{n+3-2i} & 0 \\ 
\end{array}
\right), 
\end{equation}
  where $\tr(\phi_{\zeta}^+|_{\{\omega_i,\omega_{n+3-i}\}}) = \zeta^{n} $ and 
$\tr(\psi_c^+ |_{\{\omega_i,\omega_{n+3-i}\}})=0$. 
 It follows from Corollary~\ref{cor:honest-2-dimensional-irreps} that we indeed have $2$-dimensional irreducible representations.

Furthermore, since  
 \begin{equation}
\psi_c^+(\omega_{\frac{n+3}{2}}) = -\omega_{\frac{n+3}{2}} 
\hspace{4mm} 
	\mbox{ and } 
\hspace{4mm}
 \phi_{\zeta}^+(\omega_{\frac{n+3}{2}}) 
 = - \omega_{\frac{n+3}{2}}, 
 \end{equation}
$\omega_{\frac{n+3}{2}}$ is a basis vector for a one-dimensional sign representation $\mathbb{C}\omega_{\frac{n+3}{2}}$ under the action of $D_n$ (for $n$ odd).

Similar to the proof of Theorem $7.2$ in \cite{MR3845909}, 
we have the automorphisms 
\begin{equation}\label{eqn:phi-zeta-plus-n-odd}
\phi_{\zeta}^+(\omega_i) 
= \zeta^{\frac{2n+3-2i}{2}} \omega_i    \mbox{ for } n+3 \leq i\leq 2n
\end{equation} 
and 
\begin{equation}\label{eqn:map-psi-n-odd}
\psi_c^+(\omega_{i})=-c^{n+3-2i}\omega_{n+3-i}, \mbox{ where } n+3\leq i\leq 2n. 
\end{equation} 
Since $\omega_{n+3-i}=\overline{t^{i-n-3}u\,dt}$, 
\[ 
\overline{t^{i-n-3}u\,dt} = \sum_{k=1}^{2n}P_{i-n-3,-k}\omega_k.  
\] 
Thus for $n+3\leq i\leq 2n$, we have 
\[
\psi_{c}^+(\omega_i) = -c^{n+3-2i}  \sum_{k=1}^{2n}P_{i-n-3,-k}\omega_k. 
\]

% $n=2\widetilde{m}+1$ odd   
 
The trace of $\phi_{\zeta}^+$ in basis $\{ \omega_1,\ldots,\omega_{2n}\}$ is 
$$
\tr(\phi_{\zeta}^+)= \sum_{i=1}^{2n}  \zeta^{\frac{2n+3-2i}{2}}  
$$
while the trace of $\psi_{c}^+$ with respect to $\{ \omega_1,\ldots,\omega_{2n}\}$ is  
\begin{align*}
\tr(\psi_c^+) = -1 - \sum_{i=n+3}^{2n} c^{n+3-2i}P_{i-n-3,-i}. 
%=  0 \mbox{ or } 2 or -2 ???. 
% if c is negative (the product of the roots equal -c^n), then the trace is zero. 
% if c is positive, then the trace is -2. 
\end{align*}

For $n$ odd, the system of equations we need to solve is 
  $$
  \chi_{(\Omega_R/dR)/\mathbb C\omega_0}=n_1\rho_1+n_2\rho_2+\sum_{h=1}^{\frac{n-1}{2}}m_h\chi_h, 
  $$   
  which are precisely,  
   \begin{align*}  
  2n&=n_1 +n_2   +\sum_{h=1}^{\frac{n-1}{2}}2m_h, \\  
 \sum_{i=1}^{2n}  \zeta^{\frac{(2n+3-2i)k}{2}}   &=n_1 +n_2 +\sum_{h=1}^{\frac{n-1}{2}}2m_h\cos (2\pi hk/n)\quad \mbox{ for }  
    1\leq k\leq \widetilde{m} = (n-1)/2,  \\
-1 - \sum_{i=n+3}^{2n} c^{n+3-2i}P_{i-n-3,-i} 
		&=n_1 -n_2 \quad\text{ for }\psi_c^+. \\ 
  \end{align*}
 The sum $\displaystyle{\sum_{i=1}^{2n}  \zeta^{\frac{(2n+3-2i)k}{2}}}$ simplifies as: 
 \begin{align*}
 \sum_{i=1}^{2n}  \zeta^{\frac{(2n+3-2i)k}{2}}&= \zeta^{\frac{3}{2}k}\sum_{i=1}^{2n}  \zeta^{ (n-i)k }= \zeta^{\frac{3}{2}k}\left( \zeta^{ (n-1)k }+\zeta^{ (n-2)k }+\ldots +\zeta^k+1+\zeta^{-k}+\ldots+ \zeta^{-2nk}\right)=0. 
 \end{align*}

Using Equation~\eqref{eqn:dihedral-setting-p-in-k-and-n}, 
we see that $3n-2i+4=2n-qk+1$, where $n+3\leq i\leq \frac{3(n+1)}{2}$ and $0\leq q\leq 2n/k$.  
The equation simplifies as $n+3-2i=-qk$. Since $n+3-2i$ is even, $qk$ is also even. 

So the above set of equations can be rewritten as: 
\begin{align*} 
M_{odd}  
\begin{pmatrix} 
n_1 \\ 
n_2 \\ 
m_1 \\ 
m_2 \\ 
\vdots \\ 
m_h \\ 
\vdots \\ 
m_{(n-1)/2} 
\end{pmatrix} 
= 
\begin{pmatrix}
2n \\ 
-1 - \displaystyle{\sum_{i=n+3}^{2n}} c^{n+3-2i}P_{i-n-3,-i} \\ 
  0 \\ 
  0 \\ 
  \vdots \\ 
  0 \\  
  0 \\ 
\end{pmatrix}. 
\end{align*} 
So 
\begin{align*} 
\begin{pmatrix} 
n_1 \\ 
n_2 \\ 
m_1 \\ 
m_2 \\ 
\vdots \\ 
m_h \\ 
\vdots \\ 
m_{(n-1)/2} 
\end{pmatrix}  
&= 
\tiny 
\begin{pmatrix}
   \frac{1}{2n} & \frac{n}{2n} & \frac{2}{2n} &    \ldots & \frac{2}{2n}& \ldots &  \frac{2}{2n} \\
   \frac{1}{2n} & -\frac{n}{2n}&  \frac{2}{2n} &  \ldots  & \frac{2}{2n} & \ldots &  \frac{2}{2n} \\
   \frac{2}{2n} & 0 & \frac{2}{2n} \cos\left( \frac{2\pi}{n}\right)  &\ldots &  \frac{2}{2n}\cos\left( \frac{2\pi k}{n} \right) & \ldots & \frac{2}{2n} \cos\left( \frac{2\pi \frac{n-1}{2}}{n}\right)  \\ 
   \vdots & \vdots & \vdots  & \ddots &  & \vdots & \vdots\\
   \frac{2}{2n} & 0 & \frac{2}{ 2n}\cos\left(\frac{2\pi h}{n} \right)  &\ldots &  \frac{2}{2n}\cos\left(\frac{2\pi h k}{n} \right) & \ldots & \frac{2}{2n} \cos\left(\frac{2\pi h\frac{n-1}{2}}{n} \right)  \\
   \vdots  & \vdots & \vdots & \vdots & & \ddots & \vdots \\
   \frac{2}{2n} & 0 & \frac{2}{2n}\cos\left( \frac{2\pi  \left(\frac{n-1}{2}\right)}{n}\right)   &\ldots & \frac{2}{2n}\cos\left( \frac{2\pi  k \left(\frac{n-1}{2}\right)}{n}\right)  &\ldots   & \frac{2}{2n}\cos\left( \frac{2\pi  \left(\frac{n-1}{2}\right)^2}{n}\right) \\ 
\end{pmatrix} 
\normalsize 
\begin{pmatrix}
2n \\ 
\Xi \\ 
  0 \\ 
  0 \\ 
  \vdots \\  
  0 \\ 
  0 \\ 
\end{pmatrix},  
\end{align*}
where $\Xi = -1 - \displaystyle{\sum_{i=n+3}^{2n}} c^{n+3-2i}P_{i-n-3,-i}$, 
which we conclude 
\begin{align*}
\begin{pmatrix} 
n_1 \\ 
n_2 \\ 
m_1 \\ 
m_2 \\ 
\vdots \\ 
m_h \\ 
\vdots \\ 
m_{(n-1)/2} 
\end{pmatrix} 
= 
\begin{pmatrix}
\frac{1}{2}- \frac{1}{2}\displaystyle{\sum_{i=n+3}^{2n}} c^{n+3-2i}P_{i-n-3,-i}\\ 
\frac{3}{2}+ \frac{1}{2} \displaystyle{\sum_{i=n+3}^{2n}} c^{n+3-2i}P_{i-n-3,-i}\\  
2 \\ 
2 \\   
\vdots \\   
2 \\  
\end{pmatrix}.  \\  
\end{align*} 
 
Now, let $n=2$. Then under the action by $D_2\times \mathbb{Z}_2$, we have: 
\[
\begin{aligned}
\psi_c^+(\omega_1) = -c^{\frac{3}{2}}\omega_4,  \qquad  
\psi_c^+(\omega_2) = -c^{\frac{1}{2}}\omega_3,  \qquad 
\psi_c^+(\omega_3) = -c^{-\frac{1}{2}}\omega_2, \qquad  
\psi_c^+(\omega_4) = -c^{-\frac{3}{2}}\omega_1.  
\end{aligned}
\]  
This implies with respect to the ordered basis $\{ \omega_2, \omega_3\}$ and \{$\omega_1,\omega_4 \}$, the matrix for $\psi_c^+$ is given by: 
\[ 
\psi_c^+ \Big|_{\{ \omega_2,\omega_3\}} = 
\left( 
\begin{array}{cc}
0 & -c^{-\frac{1}{2}}  \\ 
-c^{\frac{1}{2}} & 0 \\ 
\end{array}
\right) 
\qquad 
\mbox{ and } 
\qquad  
\psi_c^+ \Big|_{\{ \omega_1,\omega_4\}} = 
\left( 
\begin{array}{cc} 
0 & -c^{-\frac{3}{2}} \\ 
 -c^{\frac{3}{2}}& 0 \\ 
\end{array}
\right),  
\] 
respectively.

Under the action by $\phi_{\zeta}^+$, we have: 
\[
\begin{aligned} 
\phi_{\zeta}^+(\omega_1) = \zeta^{\frac{5}{2}}\omega_1,  \qquad 
\phi_{\zeta}^+(\omega_2) = \zeta^{\frac{3}{2}}\omega_2,  \qquad  
\phi_{\zeta}^+(\omega_3) = \zeta^{\frac{1}{2}}\omega_3,  \qquad  
\phi_{\zeta}^+(\omega_4) = \zeta^{-\frac{1}{2}}\omega_4.   
\end{aligned}
\] 
So  
\[ 
\phi_{\zeta}^+ \Big|_{\{ \omega_2,\omega_3\}} = 
\left( 
\begin{array}{cc}
\zeta^{\frac{3}{2}} & 0  \\
0& \zeta^{\frac{1}{2}}  \\
\end{array}
\right) 
\qquad 
\mbox{ and } 
\qquad 
\phi_{\zeta}^+ \Big|_{\{ \omega_1,\omega_4\}} = 
\left( 
\begin{array}{cc}
 \zeta^{\frac{5}{2}} &  0 \\
0 & \zeta^{-\frac{1}{2}}  \\
\end{array}
\right). 
\]

%Let $u_i=\zeta \omega_i$. Then 
%\[
%\begin{aligned} 
%\phi_{\zeta}^+(\omega_1) &= \zeta^{\frac{3}{2}} u_1, \\ 
%\phi_{\zeta}^+(\omega_2) &= \zeta^{\frac{1}{2}} u_2, \\ 
%\phi_{\zeta}^+(\omega_3) &= \zeta^{-\frac{1}{2}} u_3, \\ 
%\phi_{\zeta}^+(\omega_4) &= \zeta^{-\frac{3}{2}} u_4. \\ 
%\end{aligned}
%\] 
%In terms of the ordered basis $\{ u_2, u_3\}$ and $\{ u_1, u_4\}$, the matrix for $\phi_{\zeta}^+$ is given by: 
%\[ 
%\left( 
%\begin{array}{cc}
%\zeta^{\frac{1}{2}} & 0  \\
%0& \zeta^{-\frac{1}{2}}  \\
%\end{array}
%\right) \mbox{ and } 
%\left( 
%\begin{array}{cc}
% \zeta^{\frac{3}{2}} &  0 \\
%0 & \zeta^{-\frac{3}{2}}  \\
%\end{array}
%\right),  
%\]  
%respectively.  
Using the representation theory of the dihedral group $D_2$ of order $4$, we see that 
$\mathbb{C}\omega_2\oplus \mathbb{C}\omega_3$ and $\mathbb{C}\omega_1\oplus \mathbb{C}\omega_4$ are two $2$-dimensional irreducible representations of $D_2\times \mathbb{Z}_2$. 
\end{proof}

 We will now prove Corollary~\ref{cor:honest-2-dimensional-irreps} for completeness.

 \begin{proof}
% Change of basis  
For $n$ even, we change the basis to $\overline{\omega_i}= c^{-\frac{n+3-2i}{2}}\zeta^{-\frac{i}{2}}\omega_i\mbox{ for } 
1\leq i\leq n+2$ to see that we have $2$-dimensional irreducible representations. 
Since 
$$
\overline{\omega_{n+3-i}} = c^{\frac{n+3-2i}{2}}\zeta^{-\frac{n+3-i}{2}} \omega_{n+3-i} \quad \mbox{ for } 1\leq i\leq n+2, 
$$ 
we have 
\begin{enumerate}
\item 
$\phi_{\zeta}^+(\overline{\omega_i}) = c^{-\frac{n+3-2i}{2}}\zeta^{-\frac{2i}{4}}\zeta^{\frac{2n+3-2i}{2}}\omega_i 
= \zeta^{\frac{2n+3-2i}{2}}\overline{\omega_i}$, 
\item 
$\phi_{\zeta}^+(\overline{\omega_{n+3-i}}) 
= c^{\frac{n+3-2i}{2}}\zeta^{-\frac{n+3-i}{2}}\zeta^{\frac{-3+2i}{4}}\omega_{n+3-i}
= \zeta^{-\frac{2n+3-2i}{2}}\overline{\omega_{n+3-i}}
$, 
\item $\psi_{c}^+(\overline{\omega_i}) 
=-c^{-\frac{n+3-2i}{2}}\zeta^{-\frac{2i}{4}} c^{n+3-2i}\omega_{n+3-i}
=\zeta^{\frac{2n+3-2i}{2}}\overline{\omega_{n+3-i}}$, 
\item $\psi_{c}^+(\overline{\omega_{n+3-i}}) =-
c^{\frac{n+3-2i}{2}}\zeta^{-\frac{n+3-i}{2}}
c^{-\frac{n+3-2i}{2}}\omega_{i} 
= \zeta^{-\frac{2n+3-2i}{2}}\overline{\omega_i}
$. 
\end{enumerate}
%  It follows that $\overline{\omega_{n+3-i}}=c^{-\frac{3-2i+n}{4}}\zeta^{-\frac{3-2i+n}{4}}\omega_{n+3-i}$. 
With respect to the basis $\overline{\omega_1},\ldots, \overline{\omega_{n+2}}$, 
we obtain 
\[ 
\phi_{\zeta}^+\Big|_{\{ \overline{\omega_i},\overline{\omega_{n+3-i}}
\}} 
= \begin{pmatrix}
\zeta^{\frac{2n+3-2i}{2}} & 0 \\ 
0 & \zeta^{-\frac{2n+3-2i}{2}} \\ 
\end{pmatrix}
\quad \mbox{ and }\quad 
\psi_c^+ \Big|_{\{ \overline{\omega_i},\overline{\omega_{n+3-i}}
\}} = 
\begin{pmatrix}
0 & \zeta^{-\frac{2n+3-2i}{2}}\\ 
\zeta^{\frac{2n+3-2i}{2}} & 0 \\ 
\end{pmatrix}, 
\] 
which coincide with classical $2$-dimensional dihedral group irreducible representations.

For $n$ odd,  with respect to the basis $\overline{\omega_i} = c^{-\frac{n+3-2i}{4}}\zeta^{-\frac{i}{2}}\omega_i \mbox{ for } 1\leq i\leq n+2$,  we have 
\[ 
\phi_{\zeta}^+\Big|_{\{ \overline{\omega_i},\overline{\omega_{n+3-i}}
\}} = 
\begin{pmatrix}
\zeta^{\frac{2n+3-2i}{2}} & 0 \\ 
0 & \zeta^{-\frac{2n+3-2i}{2}} \\ 
\end{pmatrix}
\quad \mbox{ and }\quad 
\psi_c^+ \Big|_{\{ \overline{\omega_i},\overline{\omega_{n+3-i}}
\}} = 
\begin{pmatrix}
0 & \zeta^{-\frac{2n+3-2i}{2}}\\ 
\zeta^{\frac{2n+3-2i}{2}} & 0 \\ 
\end{pmatrix}, 
\]  
with 
 \begin{equation}
\psi_c^+(\overline{\omega_{\frac{n+3}{2}}})   
= -\overline{\omega_{\frac{n+3}{2}}}  
\quad
	\mbox{ and } 
\quad 
 \phi_{\zeta}^+(\overline{\omega_{\frac{n+3}{2}}}) 
 = \zeta^{n/2} \overline{\omega_{\frac{n+3}{2}}}
 = - \overline{\omega_{\frac{n+3}{2}}}.  
 \end{equation}
 \end{proof}

The following examples are also in \cite{MR3845909} but we state them here to provide examples for when $n$ is odd. 
\begin{example}
Let $n=3$ and $k=3$. Then $p(t)=t (t^3 -\alpha_1^3) (t^3 - \alpha_2^3)$, with 

\begin{align*}
\psi_{c}^{+}  = 
\begin{pmatrix}
0 & 0& 0& 0& -\frac{1}{c^2} & 0 \\ 
0& 0& 0& -\frac{1}{c} & 0  & 0 \\ 
0& 0&  -1 & 0 & 0 & 0 \\ 
0&  -c& 0 & 0 & 0 & 0 \\ 
 - c^2& 0& 0 & 0 & 0 & 0 \\ 
0& 0& 0 & 0 & 0 & - \frac{\alpha_1^3\alpha_2^3}{c^6} \\ 
\end{pmatrix}.
\end{align*}
In this case, $\displaystyle{\alpha_1^3\alpha_2^3=c^6}$, and the trace of $\psi_{c}^{+}$ is $-2$. This gives us multiplicities 
\begin{align*}
n_1 = n_2=n_3=0, \qquad n_4 = 2, \qquad m_1 = 2. 
\end{align*}
%but if $\displaystyle{\alpha_1^3\alpha_2^3=-c^3}$, then  the trace of $\psi_{c}^{+}$ equals $0$, giving us multiplicities  
%\begin{align*}
%n_1 = 1, n_2 = 1, m_1 = 2. 
%\end{align*}
\end{example}

\begin{example}

When $n=9$ and $k=3$, we obtain $\tr(\psi_c^+)=-2=-\tr(\psi_c^+\phi_\xi^+)$, 
and hence
$$
n_1=n_2=0,\qquad n_3=2,\qquad  n_4=4,\qquad  m_1=6,\qquad  m_2=0.
$$

\end{example}

\noindent {\bf Acknowledgments.}  B.C. is partially supported by a collaboration grant from the Simons Foundation $\#319261$. 
X.G. is partially supported by NSF of China
(Grants $11101380$, $11471294$) and the Foundation for Young Teachers of
Zhengzhou University (Grant $1421315071$).   K.Z. is partially
supported by NSF of China (Grant $11271109$) and NSERC.

%  \bibliography{infinite-dim-Lie-algs} 

%\def\cprime{$'$} \def\cprime{$'$} \def\cprime{$'$} \def\cprime{$'$}
\providecommand{\bysame}{\leavevmode\hbox to3em{\hrulefill}\thinspace}
\providecommand{\MR}{\relax\ifhmode\unskip\space\fi MR }
\providecommand{\MRhref}[2]{%
  \href{http://www.ams.org/mathscinet-getitem?mr=#1}{#2}
}
\providecommand{\href}[2]{#2}

\appendix

\end{document}